\documentclass[oneside, a4paper]{amsart}
\usepackage{amsmath}
\usepackage{amssymb}
\usepackage{amsthm}
\usepackage{graphicx}
\usepackage{caption}
\usepackage{tikz}
\usepackage{amsmath,amssymb,amsfonts}
\usepackage{mathtools}
\usepackage{enumerate}
\usepackage{enumitem}
\usepackage{doi}
\usepackage{faktor}
\usepackage[font=small,labelfont=bf]{caption}
\usepackage{eucal}
\usepackage{mathrsfs}
\usepackage{stmaryrd}
\usepackage{verbatim}
\usepackage{manfnt}
\usepackage{array}
\usetikzlibrary{patterns,cd}
\usepackage{layout} 
\usepackage[all]{xy}
\usepackage{faktor}

\numberwithin{equation}{section}

\begingroup
\catcode`\&=13
\gdef\pampmatrix{%
  \begingroup
  \let&=\amsamp
  \begin{pmatrix}%
}
\gdef\endpampmatrix{\end{pmatrix}\endgroup}
\endgroup

\newcommand{\HH }{\mathrm{HH}}
\newcommand{\Set }{\mathrm{Set}}
\newcommand{\im}{\textnormal{im}}

\newcommand{\Z}{\mathbb{Z}}
\newcommand{\C}{\mathbb{C}}
\newcommand{\N}{\mathbb{N}}
\newcommand{\Q}{\mathbb{Q}}
\newcommand{\Hom}{\operatorname{Hom}}

\newcommand{\co}{\colon\thinspace}
\newcommand{\Fin}{\mathrm{Fin}}
\newcommand{\Topp}{\mathrm{Top}}
\newcommand{\dgca}{\mathrm{dgca}}
\newcommand{\sset}{\mathrm{sSet}}

\theoremstyle{plain}

\newtheorem{theorem}{Theorem}[section]
\newtheorem*{theorem*}{Theorem}
\newtheorem{proposition}[theorem]{Proposition}
\newtheorem{lemma}[theorem]{Lemma}
\newtheorem{addendum}[theorem]{Addendum}
\newtheorem{corollary}[theorem]{Corollary}

\theoremstyle{definition}
\newtheorem{definition}[theorem]{Definition}
\newtheorem{example}[theorem]{Example}
\newtheorem{construction}[theorem]{Construction}

\theoremstyle{remark}
\newtheorem{remark}[theorem]{Remark}

\usepackage[all]{xy}
\SelectTips{cm}{10}
\CompileMatrices

\title[The equivariant cohomology ring of $\Hom(\Z^2,\mathrm{GL}_n(\C))$]{The equivariant cohomology ring of the representation variety $\Hom(\Z^2,\mathrm{GL}_n(\C))$}
\author{Simon Gritschacher}
\address{Mathematisches Institut der Ludwig-Maximilians-Universit{\"a}t  \\ Theresienstrasse 39, 80333, Munich, Germany \\ Email: \url{simon.gritschacher@math.lmu.de}}
\date{\today}

\begin{document}

\maketitle

\vspace{-10pt}

\begin{abstract}
We give a presentation of the $\mathrm{GL}_n(\C)$-equivariant cohomology ring with $\Z$-coefficients of the variety $\Hom(\Z^2,\mathrm{GL}_n(\C))\subseteq \mathrm{GL}_n(\C)^2$ for any $n$. It is torsion free and minimally generated as a $H^\ast B\mathrm{GL}_n(\C)$-algebra by $3n$ elements. The ideal of relations is the saturation of an $n$-generator ideal by even powers of the Vandermonde polynomial. For coefficients in a field whose characteristic does not divide $n!$, we also give a presentation of the non-equivariant cohomology ring of $\Hom(\Z^2,\mathrm{GL}_n(\C))$.
\end{abstract}

\setcounter{tocdepth}{1}
\tableofcontents

\vspace{-10pt}

\section{Introduction}

We will compute the equivariant cohomology ring, with integral coefficients, of the representation variety $\Hom(\Z^2,\mathrm{GL}_n(\C))$. By a result of Pettet and Souto \cite{PS13} this space is homotopy equivalent to the representation space $\Hom(\Z^2,U(n))$. From the viewpoint of geometry, the cohomology we compute may be identified with the cohomology of the moduli stack of flat holomorphic vector bundles over an elliptic curve, or of flat unitary connections on a torus. Moduli spaces of flat connections on compact surfaces have been studied extensively, and determining the ring structure on their cohomology is a central and often difficult problem. For surface groups of genus $g\geq 2$ one often restricts attention to smooth components of the moduli space by imposing constraints on rank and degree (see e.g. the pioneering paper of Atiyah and Bott \cite{AB}, the survey of Jeffery \cite{JeffreySurvey}, or \cite{hausel, jeffreykirwan, kirwan}). In genus $1$, there is only a single component which is singular when $n\geq 2$ due to the presence of reducible connections. The approach we take is by homotopy theoretic methods (Hochschild homology, $K$-theory) rather than differential geometric ones.

A representation $\Z^2\to \mathrm{GL}_n(\C)$ is determined by the images of the generators, and thus $\Hom(\Z^2,\mathrm{GL}_n(\C))$ may be identified with the commuting variety 
\[
C_2(\mathrm{GL}_n(\C))=\{(A,B) \in \mathrm{GL}(n,\C)^2\mid AB=BA\}\, ,
\]
equipped with the $\mathrm{GL}_n(\C)$-action by simultaneous conjugation. Baird \cite{Baird} has proved an appealing formula for the cohomology ring of the space of $m$-tuples of commuting elements in a (compact) Lie group $G$. The description is given in terms of invariants for the action of the Weyl group on a maximal torus. In the present situation this formula implies
\begin{equation} \label{eq:bairdintro}
H^\ast_{\mathrm{GL}_n(\C)}(C_2(\mathrm{GL}_n(\C));\Q) \cong H^\ast(BT\times T^2;\Q)^{S_n}\, ,
\end{equation}
where $T$ is an $n$-dimensional torus and $S_n$ is the symmetric group. However, this description is valid only over coefficient fields of characteristic prime to $|S_n|$, and it provides no information about the structure of the ideal of relations. Using Baird's formula, Kishimoto and Takeda \cite{KT21} identified minimal generating sets for the rational cohomology rings of spaces of commuting elements in various classical groups, but the full ring structure was made explicit only by ad hoc computation in a very limited number of cases; mostly Lie groups of rank $\leq 2$ (see e.g. \cite{Fok, KT21, Takeda}).

In this paper we fully solve both the question of torsion and the determination of relations in the equivariant cohomology ring of $C_2(\mathrm{GL}_n(\C))$. We show in particular that the above formula (\ref{eq:bairdintro}) does not hold integrally, although the equivariant cohomology ring embeds as a proper subring of the corresponding invariant ring.

\subsection{Our main theorem} Let $t_1,\dots,t_n$ be variables of degree two and let $\Lambda_n=\Z[t_1,\dots,t_n]^{S_n}$ denote the ring of symmetric polynomials in $t_1,\dots,t_n$. We identify $\Lambda_n$ with the cohomology ring of $B\mathrm{GL}_n(\C)$; under this identification the Chern classes correspond to the elementary symmetric polynomials.  Let $V=(t_j^{i-1})_{1\leq i,j\leq n}$ be the Vandermonde matrix and $\Delta=\prod_{i < j}(t_j-t_i)$ its determinant. Since $\Delta$ is alternating, its square $\Delta^2$ is a symmetric polynomial. Let $\mathrm{adj}(V)=\Delta V^{-1}\in \mathrm{Mat}_n(\Z[t_1,\dots,t_n])$ be the adjugate of $V$. Let $s$ be another variable. Define the diagonal matrix $D(s)=\mathrm{diag}(s/(1+t_i s))_{1\leq i \leq n}$, and the matrix
\[
M(s)=\mathrm{adj}(V)^T D(s)\mathrm{adj}(V)\,;
\]
this matrix is invariant under permutation of $t_1,\dots,t_n$. Thus, for any two $k$-element subsets $I,J\subseteq \{1,\dots,n\}$ the $(I,J)$-minor of $M(s)$ defines a formal power series
\[
\det(M(s)_{I,J})\in \Lambda_n [[s]]\, .
\]
It turns out that $\det(M(s)_{I,J})$ is divisible by $\Delta^{2k-2}$, and we define the symmetric polynomial $c_{I,J}^l\in \Lambda_n$ as the coefficient of $s^l$ in $\det(M(s)_{I,J})/\Delta^{2k-2}$.

Let $Z_1,\dots,Z_n$ be variables of degree $|Z_i|=2i$, and let $X_1,\dots,X_n,Y_1,\dots,Y_n$ be variables of degree $|X_i|=|Y_i|=2i-1$. We consider the free graded commutative $\Lambda_n$-algebra
\[
A:=\Lambda_n\otimes \Z[Z_1\dots,Z_n]\otimes {\textstyle \bigwedge_\Z}(X_1,\dots,X_n,Y_1,\dots,Y_n)\, .
\]
For a subset $I=\{i_1,\dots,i_k\}\subseteq \{1,\dots,n\}$ with $i_1<\cdots < i_k$ we write $X_I$ for the product $X_{i_1}\cdots X_{i_k}$ in $A$, and similarly $Y_I$ for the product $Y_{i_1}\cdots Y_{i_k}$.

For $1\leq l \leq n$ let $R_l \in A$ denote the element
\begin{equation} \label{eq:idealofrelations}
\Delta^2 Z_l-\sum_{k=1}^l (-1)^{\frac{k(k-1)}{2}}\sum_{\substack{I,J\subseteq \{1,\dots,n\}\\ |I|=|J|=k}} c^l_{I,J} X_I Y_J\, ,
\end{equation}
which is homogeneous of degree $2(n(n-1)+l)$. The elements $R_1,\dots,R_n$ generate an ideal $(R_1,\dots,R_n)\subseteq A$. We consider its saturation with respect to $\Delta^2$; this is the ideal $\mathcal{J}\subseteq A$ defined by $\mathcal{J}=\{a\in A\mid \exists m\geq 0: \Delta^{2m}a\in (R_1,\dots,R_n)\}$.

\begin{theorem} \label{thm:main}
For every $n\geq 0$ there is an isomorphism of $\Lambda_n$-algebras
\[
H^\ast_{\mathrm{GL}_n(\C)}(C_2(\mathrm{GL}_n(\C));\Z) \cong \faktor{A}{\mathcal{J}} \, ,
\]
where $\mathcal{J}$ is the saturation of the ideal $(R_1,\dots,R_n)$ with respect to $\Delta^2$.

Moreover, this algebra cannot be generated by fewer than $3n$ elements.
\end{theorem}

Our proof will show that the generators are pulled back from the direct limit $C_2(\mathrm{GL}_\infty(\C))=\mathrm{colim}_n C_2(\mathrm{GL}_n(\C))$ and can be interpreted in terms of Chern classes (see Remark \ref{rem:kishimototakeda}). A few immediate consequences of Theorem \ref{thm:main} are:

\begin{enumerate}
\item[(1)] After inverting $\Delta^2$, the equivariant cohomology of $C_2(\mathrm{GL}_n(\C))$ is a free $\Lambda_n[\Delta^{-2}]$-module of total rank $2^{2n}$; in fact, isomorphic to
\[
\Lambda_n[\Delta^{-2}]\otimes {\textstyle \bigwedge_{\Z}}(X_1,\dots,X_n,Y_1,\dots,Y_n)\,.
\]
\item[(2)] The equivariant cohomology of $C_2(\mathrm{GL}_n(\C))$ is $\Lambda_n$-torsionfree, in particular torsionfree as an abelian group. This follows from the previous item and the fact that the ideal $\mathcal{J}$ is $\Delta^2$-saturated, so there is no $\Delta^2$-torsion.
\item[(3)] Let $T$ be a maximal torus of the unitary group $U(n)\subseteq \mathrm{GL}_n(\C)$. Then the $T$-equivariant cohomology of $C_2(\mathrm{GL}_n(\C))$ is $H^\ast(BT;\Z)$-torsionfree. This follows from the previous bullet point and the fact that $\Lambda_n$ is a free $H^\ast(BT;\Z)$-module.
\end{enumerate}

The non-equivariant cohomology of $\Hom(\Z^2,\mathrm{GL}_n(\C))$ has intricate torsion, and the map from the equivariant to the non-equivariant cohomology is not surjective in general. This is reflected in the fact that the equivariant cohomology is not a free $\Lambda_n$-module with $\Z$-coefficients, although it is free with $\Q$-coefficients.  A central motivation for this work is to develop an approach to understanding this torsion; for example, through the Eilenberg-Moore spectral sequence which requires knowledge of the equivariant cohomology as a $\Lambda_n$-module.

It is natural to ask to what extent our results generalise beyond $\mathrm{GL}_n(\C)$. While we do not know whether the equivariant cohomology of $\Hom(\Z^2,G)$ is torsionfree for arbitrary $G$, the algebraic structure of the cohomology ring does generalise in the non-modular case. In this setting our results can be interpreted as an explicit presentation of the invariant algebra appearing on the right hand side of (\ref{eq:bairdintro}). The Vandermonde determinant $\Delta$ identifies with the $T$-equivariant Euler class of $U(n)/T$; for other compact Lie groups $G$ this role is played by the Euler class of $G/T$. The matrices encoding the relations are replaced by higher order Hessians of the basic invariants in $H^\ast(BT)^W$. These generalisations will be treated elsewhere.

\subsection{Outline of the paper} We briefly outline the proof strategy and the structure of the paper. Since $C_2(\mathrm{GL}_n(\C))$ deformation retracts onto the space of commuting pairs in the unitary group $U(n)\subseteq \mathrm{GL}_n(\C)$, we will compute the $U(n)$-equivariant cohomology ring of $C_2(U(n))=\{(A,B)\in U(n)^2\mid AB=BA\}$.

In Section \ref{sec:notorsion} we set up a spectral sequence which starts from iterated Hochschild homology of a polynomial algebra and abuts on the equivariant homology of $C_2(U(n))$ (Proposition \ref{prop:ss}). 

The spectral sequence is used to show that the inclusion of a maximal torus $T\hookrightarrow U(n)$ embeds the equivariant cohomology of $C_2(U(n))$ into the free graded commutative algebra $H^\ast(BT\times T^2;\Z)$ (Corollary \ref{cor:embedding}). 

In Section \ref{sec:generators} the spectral sequence is used to show that the map into the direct limit $C_2(U(n))_{hU(n)} \to C_2(U)_{hU}$ is surjective in cohomology. This gives us a convenient set of generators, since homotopically $C_2(U)_{hU}$ is just the product $BU^2\times U^2$ (Corollary \ref{cor:gcsurjective}).

In Proposition \ref{prop:3ngenerators} we prove that $3n$ generators are enough. Their images under the embedding into a free graded commutative algebra are computed in (\ref{eq:explicitformxy}) and (\ref{eq:explicitformz}). At this point we have a description of the equivariant cohomology ring as an explicit subalgebra of a free graded commutative algebra.

In Section \ref{sec:ideal} formal algebraic manipulations are made to identify the ideal of relations $\mathcal{J}$ of Theorem \ref{thm:main} (Lemma \ref{lem:ideal}).

In Section \ref{sec:minimality} we finish the proof of the main theorem, by showing that the exhibited generating set is minimal. In Addendum \ref{add:minimal} we show that after inverting $n$, the size of a minimal generating set is $3n-1$ instead of $3n$.

In Section \ref{sec:inverting} we determine a presentation of the non-equivariant cohomology ring of $C_2(\mathrm{GL}_n(\C))$ with coefficients in a field of characteristic zero or prime to $n!$ (Theorem \ref{thm:main2}). An algorithm which computes an additive basis of the ideal of relations is given in Corollary \ref{cor:main2}.

In Section \ref{sec:ab} we give a gauge theoretic interpretation of our result, by commenting on the surjectivity of the natural map, considered by Atiyah--Bott \cite{AB}, from the cohomology of the classifying space of the gauge group to the equivariant cohomology of $C_2(U(n))$.

There is an Appendix \ref{sec:appendix} which explains some details about the spectral sequence we use.

\section{A spectral sequence and the absence of torsion} \label{sec:notorsion}

\subsection{Preliminaries}
Our aim is to compute the equivariant cohomolgy ring of $C_2(\mathrm{GL}_n(\C))$. A few intermediate results will be proved for the more general spaces $C_m(\mathrm{GL}_n(\C))$, the space of $m$-tuples of pairwise commuting elements in $\mathrm{GL}_n(\C)$. Clearly, $\mathrm{GL}_n(\C)$ deformation retracts onto the unitary group $U(n)$, which is its maximal compact subgroup. What is far less evident -- and a nontrivial theorem of Pettet--Souto \cite{PS13} -- is that the same is true for the commuting varieties: $C_m(\mathrm{GL}_n(\C))$ deformation retracts onto $C_m(U(n))$ for all $m$. The spaces $C_m(U(n))$ are more tractable than the general linear counterparts, which is why all of our computations will be done for $C_m(U(n))$.

For convenience, we recall the definition:
\[
C_m(U(n))=\{(A_1,\dots,A_m)\in U(n)^m \mid \forall i,j : A_iA_j=A_jA_i\}\, .
\]
We let $U(m)$ act on $C_n(U(m))$ by simultaneous conjugation and denote by
\[
C_m(U(n))_{hU(n)}=EU(n)\times_{U(n)} C_m(U(n))
\]
the homotopy orbit space; by definition, its cohomology is the $U(n)$-equivariant cohomology of $C_m(U(n))$. We note that the homotopy equivalence between $C_m(U(n))$ and $C_m(\mathrm{GL}_n(\C))$ carries over to the homotopy orbits, that is
\[
C_m(U(n))_{hU(n)} \simeq C_m(\mathrm{GL}_n(\C))_{h\mathrm{GL}_n(\C)}\, .
\]
Therefore, the corresponding equivariant cohomology rings are isomorphic, and all of our computations apply equally to the unitary and the complex case. Our main results will be stated for $C_m(\mathrm{GL}_n(\C))$, because this space is of more interest in algebraic geometry, but this is really a matter of taste.

An important perspective is to consider the spaces $C_m(U(n))$ simultaneously for all $n$, as this reveals some additional structure which is otherwise hidden. We define
\begin{equation}
\label{eq:cm}
\mathscr{C}_m=\bigsqcup_{n\geq 0} C_m(U(n))_{hU(n)}\,.
\end{equation}
One can view $\mathscr{C}_m$ as the classifying space of a topological symmetric monoidal category (namely, finite dimensional unitary representations of $\Z^m$ under direct sum); it is therefore an $\mathbb{E}_\infty$-space, that is, a coherently homotopy commutative $H$-space. We will use this structure in a couple of places without much further explanation; a good reference is \cite{Lawson} or \cite{RModuli}.

Another fact that we will use sometimes is that $C_m(U(n))$ is path-connected for all $m,n\geq 0$. This follows because in the unitary groups any abelian subgroup is contained in a maximal torus.

Our first goal is to set up a spectral sequence which starts form iterated Hochschild homology of a polynomial algebra and converges to the homology of $\mathscr{C}_m$.

\subsection{Loday construction} \label{sec:loday}

We recall the Loday construction and higher Hochschild homology, see e.g. \cite[Section 1.7]{Pirashvili}. Let $k$ be a commutative ring and let $\dgca_k$ be the category of differential graded commutative algebras over $k$. Let $\Fin, \Set$ and $\Fin_\ast, \Set_\ast$ be the categories of (finite) sets and (finite) pointed sets, respectively. The basepoint of a based set $S\in \Set_\ast$ will be denoted by $\ast\in S$; the finite based set $\{\ast,1,\dots,n\}$ will be denoted by $\langle n\rangle$. The one-point set (or space) will also be denoted by $\ast$.

\begin{construction}[e.g. \cite{Pirashvili}]
Fix $A\in \dgca_k$. Since $\dgca_k$ has all colimits and the category $\Set$ is generated under colimits by $\ast$, we can make the following construction:
\begin{enumerate}
\item[(1)] We let $\mathcal{L}(A)\co \Set \to \dgca_k$ be the unique (up to natural isomorphism) colimit preserving functor satisfying $\mathcal{L}(A)(\ast)=A$.
\item[(2)] If $S\in \Set_\ast$, then $\mathcal{L}(A)(S)$ is canonically a dgca under $A$ (i.e., equipped with a canonical map $A\to \mathcal{L}(A)(S)$). Suppose that $A$ is augmented over $k$. Then we define $\mathcal{L}(A;k)\co \Set_\ast \to \dgca_k$ by $\mathcal{L}(A;k)(S)=k\otimes_A \mathcal{L}(A)(S)$.
\end{enumerate}
\end{construction}

Concretely, for $\{1,\dots,n\} \in \Fin$
\[
\mathcal{L}(A)(\{1,\dots,n\})\cong A^{\otimes n}
\]
is the choice of an $n$-fold tensor product (over $k$) of $A$ with itself. For $S\in \mathrm{Set}$
\[
\mathcal{L}(A)(S)\cong \mathrm{colim}_{\substack{U\subseteq S\\ \textnormal{finite}}}\, \mathcal{L}(A)(U)\,,
\]
where the colimit runs over all finite subsets of $S$. If $S$ is a pointed set and $A$ is augmented over $k$, then $\mathcal{L}(A;k)(S)$ is the pushout in $\dgca_k$ of the diagram
\[
\xymatrix{
A\ar[r] \ar[d] & \mathcal{L}(A)(S) \\
k
}
\]
It is again augmented over $k$, since $\mathcal{L}(A)(S)$ is always augmented over $A$ (via the unique map $S\to \ast$).

Let $\sset$ (respectively, $\sset_\ast$) denote the category of (based) simplicial sets. For $X\in \sset$, evaluating $\mathcal{L}(A)$ on $X$ levelwise gives a simplicial dgca
\[
\mathcal{L}(A)(X):=\{[q]\mapsto \mathcal{L}(A)(X_q)\}\, .
\]
This simplicial dgca can be viewed as a bidifferential bigraded $k$-module
\[
(\mathcal{L}(A)(X),d_h,d_v)
\]
with a ``horizontal" differential $d_h\co \mathcal{L}(A)(X_q)_p\to \mathcal{L}(A)(X_q)_{p-1}$ coming from $A$, and a ``vertical" differential $d_v\co \mathcal{L}(A)(X_q)_p\to \mathcal{L}(A)(X_{q-1})_p$ given as usual by the alternating sum of the face maps of the simplicial dgca. We denote the total complex of this bicomplex by
\[
|\mathcal{L}(A)(X)|=\mathrm{Tot}^{\oplus}(\mathcal{L}(A)(X))\, .
\]
For $X\in \sset_\ast$, $|\mathcal{L}(A;k)(X)|$ is defined analogously.

\begin{example} \label{ex:ordinaryhhcomplex} Let $A$ be a graded commutative $k$-algebra (with trivial differential). Let $S^1$ be the standard simplicial model of the circle with only two non-degenerate simplices in total. Then $|\mathcal{L}(A)(S^1)|$ is the usual Hochschild complex \cite[§ 1.1.1]{Lo92}.  In (simplicial) degree $n$ it is the tensor product $A^{\otimes (n+1)}$ and the differential $d\co A^{\otimes (n+1)}\to A^{\otimes n}$ is given by
\begin{equation*}
\begin{split}
d(a_0\otimes \cdots \otimes a_n) & =(-1)^{n+(|a_0|+\cdots +|a_{n-1}|) |a_n|)} a_n a_0 \otimes \cdots \otimes a_{n-1}\\&\quad +\sum_{i=0}^{n-1} (-1)^i a_0\otimes \cdots \otimes a_{i}a_{i+1}\otimes \cdots \otimes a_n\, .
\end{split}
\end{equation*}
\end{example}

\begin{definition} \label{def:hh}
Let $X\in \sset$ and $A$ a dgca over a commutative ring $k$.
\begin{enumerate}
\item[(1)] The (underived) \emph{Hochschild homology of $A$ over $X$} is defined by
\[
\HH^X_\ast(A)=H_\ast(|\mathcal{L}(A)(X)|)\,.
\]
\item[(2)] If $X$ is based and $A$ is augmented over $k$, then the (underived) Hochschild homology \emph{with coefficients in the $A$-module $k$} is defined by
\[
\HH_\ast^X(A;k)=H_\ast(|\mathcal{L}(A;k)(X)|)\,.
\]
\end{enumerate}
\end{definition}

Let $X\in \sset$ and denote by $X_+$ the simplicial set $X$ with a disjoint basepoint added. Then, if $A$ is augmented over $k$, there is a natural isomorphism
\begin{equation} \label{eq:lodayextrabasepoint}
\mathcal{L}(A;k)(X_+) \cong \mathcal{L}(A)(X)\, ;
\end{equation}
this follows from the definition of $\mathcal{L}(A;k)(X_+)$ and the isomorphism
\[
\mathcal{L}(A)(X_+)\cong A\otimes \mathcal{L}(A)(X)
\]
under $A$ (recall that $\mathcal{L}(A)$ preserves colimits and $X_+=X\sqcup \ast$).

Both $|\mathcal{L}(A)(X)|$ and $|\mathcal{L}(A;k)(X)|$ are again dgcas over $k$ under the \emph{shuffle product} (see \cite[\S 4.2]{Lo92} in the classical context; Definition \ref{def:shuffle} for the augmented case; the shuffle product for $\mathcal{L}(A)(X)$ is defined analogously). Thus, the shuffle product gives $\HH_\ast^X(A)$ (and $\HH_\ast^X(A;k)$ in the based/augmented case) the structure of a graded commutative $k$-algebra. 

\begin{example} \label{ex:classicalhh}
Let $k$ be a commutative ring. For simplicity, assume that $2\neq 0$ in $k$, so that, if $y$ is an odd degree variable, the exterior algebra ${\textstyle \bigwedge_k(y)=k[y]/(y^2)}$ is a free graded commutative algebra. Let $x$ be a generator of even degree and $\sigma x$ a generator of degree $|\sigma x|=|x|+1$. Define a map of graded commutative $k$-algebras
\begin{equation} \label{eq:hhpolynomial}
k[x]\otimes {\textstyle \bigwedge_{k}}(\sigma x) \to |\mathcal{L}(k[x])(S^1)|
\end{equation}
by sending $x\mapsto x$ (in simplicial degree $0$) and $\sigma x\mapsto 1\otimes x$ (in simplicial degree $1$). It is easily checked (see Example \ref{ex:ordinaryhhcomplex}) that $1\otimes x$ is a cycle, and thus (\ref{eq:hhpolynomial}) is a map of dgca's, if the algebra appearing on the left hand side is given the zero differential. It is a classical fact (for example, a consequence of the HKR-Theorem \cite[Theorem 3.4.4]{Lo92}) that this map is an isomorphism on homology. This shows not only that
\[
\HH_\ast^{S^1}(k[x]) \cong k[x]\otimes {\textstyle \bigwedge_{k}}(\sigma x)
\]
as graded algebras, but it also shows that the Loday construction $\mathcal{L}(k[x])(S^1)$ is formal (cf. \cite[Proposition 2.1]{BLPRZ15}). This will be important in Corollary \ref{cor:hhiterated} below.
\end{example}

\begin{example} \label{ex:classicalhh2}
Now suppose that $x$ is a generator of odd degree, and $\sigma x$ a generator of degree $|\sigma x|=|x|+1$. Then there is an isomorphism of graded algebras
\[
\HH_\ast^{S^1}({\textstyle \bigwedge_k}(x))\cong {\textstyle \bigwedge_k}(x)\otimes \Gamma_k[\sigma x]\,,
\]
where $\Gamma_k[\sigma x]$ is a divided polynomial algebra. The computation goes as follows: Let $A=\bigwedge_k(x)$. The Hochschild complex $|\mathcal{L}(A)(S^1)|$  is quasi-isomorphic to the normalised Hochschild complex (see \cite[\S 1.1.14]{Lo92}) which has
\[
A\otimes \bar{A}^{\otimes n} \cong k\{1\otimes x^{\otimes n}, x^{\otimes (n+1)}\}
\]
in degree $n$; here $\bar{A}=(x) \subseteq A$. We have that $d(x^{\otimes (n+1)})=0$, because $x^2=0$, and $d(1\otimes x^{\otimes n})=(-1)^{n+(n-1)|x|^2}+1=0$, because $|x|$ is odd. It follows that the map $\bigwedge_{k}(x) \otimes \Gamma_k[\sigma x] \to \HH_\ast^{S^1}(A)$ which sends $x\mapsto x \in A\otimes \bar{A}^{\otimes 0}=A$ and $\gamma_m(x)\mapsto 1\otimes x^{\otimes m} \in A\otimes \bar{A}^{\otimes m}$ is an isomorphism. By inspection, it is multiplicative with respect to the shuffle product on the normalised Hochschild complex. In fact, since the shuffle product descends to the normalised Hochschild chains, the computation gives a zig-zag of quasi-isomorphisms of dgca's between $|\mathcal{L}(\bigwedge_k(x))(S^1)|$, the normalised Hochschild complex and $\bigwedge_k(x) \otimes \Gamma_k[\sigma x]$ (cf. \cite[Proposition  2.3]{BLPRZ15}).
\end{example}

In the following, we let $S^1$ denote the standard simplicial circle and $(S^1)^m$ the product simplicial set. Under suitable flatness hypotheses on $A$, the Loday construction $\mathcal{L}(A)(X)$ depends, up to quasi-isomorphism, only on the weak homotopy type of $X$; we will not have to appeal to this fact.

We shall use, however, the fact that a quasi-isomorphism $A\to B$ of bounded below, degreewise flat dgca's induces a quasi-isomorphism $\mathcal{L}(A)(X) \to \mathcal{L}(B)(X)$. This follows, because tensoring of flat dgca's preserves quasi-isomorphisms, and a levelwise quasi-isomorphism of simplicial dgca's induces a quasi-isomorphism on total complexes (by arguing with the standard spectral sequence for the homology of a total complex; we use boundedness for the convergence of the spectral sequence).

We will also use the natural isomorphism
\begin{equation} \label{eq:lodayoftensorproduct}
\mathcal{L}(A\otimes B)(X) \cong \mathcal{L}(A)(X) \otimes \mathcal{L}(B)(X)
\end{equation}
for any two dgca's $A$ and $B$ and any simplicial set $X$; to see this it suffices to check it on finite pointed sets, where it follows directly from the construction.

\begin{corollary} \label{cor:hhiterated}
Let $x$ be a variable of even degree. Then, there is an isomorphism
\[
\HH_\ast^{(S^1)^2}(\Z[x]) \cong \Z[x] \otimes {\textstyle \bigwedge_\Z}(\sigma x,\tau x) \otimes \Gamma_{\Z}[\rho x]\,,
\]
where the degrees are $|\sigma x|=|\tau x|=|x|+1$ and $|\rho x|=|x|+2$.
\end{corollary}
\begin{proof}
The Eilenberg-Zilber theorem gives a quasi-isomophism of dgca's
\[
|\mathcal{L}(\Z[x])((S^1)^2)| \simeq \mathcal{L}(|\mathcal{L}(\Z[x])(S^1)|)(S^1)
\]
(see, for example, \cite[Corollary 2.4.4]{GTZ10}). By Example \ref{ex:classicalhh} and using the fact that a quasi-isomorphism of flat dgcas induces a quasi-isomorphism of Loday constructions, we get a quasi-isomorphism $\mathcal{L}(|\mathcal{L}(\Z[x])(S^1)|)(S^1) \simeq |\mathcal{L}(\Z[x] \otimes \bigwedge_\Z(\sigma x))(S^1)|$. Now the claim follows from the isomorphism (\ref{eq:lodayoftensorproduct}), the K{\"u}nneth isomorphism and the computation of Example \ref{ex:classicalhh2}.
\end{proof}

\subsection{The spectral sequence} \label{sec:ss}

Let $\mathscr{C}_m$ be as defined in (\ref{eq:cm}).

\begin{proposition} \label{prop:ss}
Let $x_0,x_1,\dots$ be variables of degree $|x_i|=2i$. For every $m\geq 0$ and any commutative ring $k$, there is a first quadrant spectral sequence of algebras
\[
E^2=\HH_\ast^{(S^1)^m}(k[x_0,x_1,x_2,\dots]) \Longrightarrow H_\ast(\mathscr{C}_m;k)\, .
\]
If $k$ is a field of characteristic zero, the spectral sequence degenerates at the $E^2$-page.
\end{proposition}

The bigrading of the $E^2$-page is as follows: the vertical degree is the ``internal" degree coming from the grading of the variables $x_i$, and the horizontal degree, let's say $p\geq 0$, is the ``simplicial" degree indexing the Hochschild homology groups $\HH_p^{(S^1)^m}$.

\begin{proof}
The spectral sequence is a special case of the spectral sequence we obtained in \cite[Proposition 5.5]{Grit25}. But we supply more details: Let $ku$ be the $\Gamma$-space for complex K-theory (see Example \ref{ex:ku}). First, one shows that $\mathscr{C}_m$ is homotopy equivalent to $ku((S^1)^m_+)$, where $(S^1)^m_+=\Hom(\Z^m,S^1)_+$ is the character group of $\Z^m$ with a disjoint basepoint added. This is a special case of \cite[Proposition 3.6]{Grit25}, and an almost identical statement is in \cite[Proposition 2.4]{GH19}.

The underlying space of $ku$ is $\bigsqcup_{n\geq 0}BU(n)$, whose homology is the polynomial algebra $k[x_0,x_1,\dots]$ on the duals of the powers of the first Chern class. It follows that $H_\ast(ku(\langle 1\rangle);k)$ is $k$-flat. Since $ku$ is special and cofibrant, Proposition \ref{prop:ssgeneral} gives a spectral sequence of algebras
\[
E^2=\HH^{(S^1)^m_+}(k[x_0,x_1,x_2,\dots];k) \Longrightarrow H_\ast(\mathscr{C}_m;k)\, .
\]
This becomes the claimed spectral sequence by using the isomorphism (\ref{eq:lodayextrabasepoint}).

Now suppose that $k$ is a field of characteristic zero. One shows that the spectral sequence degenerates by showing that the $E^2$-page is isomorphic to the abutment: The $E^2$-page can be computed iteratively, just like in Corollary \ref{cor:hhiterated}, by using the computations of Examples \ref{ex:classicalhh} and \ref{ex:classicalhh2}, the fact that over a field of characterstic zero a divided polynomial algebra is isomorphic to a polynomial algebra, and the fact that a free graded commutative algebra is intrinsically formal. For example,
\begin{equation*}
\begin{split}
\HH_\ast^{(S^1)^m}(k[x_0]) & \cong H_\ast(\mathcal{L}(|\mathcal{L}(k[x_0])(S^1)|)((S^1)^{m-1})) \\& \cong H_\ast(\mathcal{L}(k[x_0] \otimes {\textstyle \bigwedge_k(\sigma x_0)})((S^1)^{m-1})) \\&\cong \HH_\ast^{(S^1)^{m-1}}(k[x_0]) \otimes \HH_\ast^{(S^1)^{m-1}}({\textstyle \bigwedge_k}(\sigma x_0)) \,,
\end{split}
\end{equation*}
and further,
\begin{equation*}
\begin{split}
\HH_\ast^{(S^1)^{m-1}}({\textstyle \bigwedge_k}(\sigma x_0)) \cong \HH_\ast^{(S^1)^{m-2}}({\textstyle \bigwedge_k}(\sigma x_0)) \otimes \HH_\ast^{(S^1)^{m-2}}(\Gamma_{k}[\sigma^2 x_0]) \,.
\end{split}
\end{equation*}
Now use $\Gamma_{k}[\sigma^2 x_0]\cong k[\sigma^2 x_0]$, and continue by induction. Let $\mathscr{S}_\ast$ be a graded set with $\binom{m}{l}$ elements in degree $l$. Then $\HH_\ast^{(S^1)^m}(k[x_0])$ is a free graded commutative $k$-algebra generated by $\mathscr{S}_\ast$. Accordingly, repeating this argument with the variables $x_i$ of degree $2i$, the $E^2$-page is a free graded commutative $k$-algebra generated by the graded set $\bigsqcup_{i\geq 0} \Sigma^{2i}\mathscr{S}_{\ast}$, where $\Sigma^{2i}$ is a shift in grading by $2i$.

On the other hand, as explained in \cite[Section 6.2]{Grit25}, if $k$ is a characteristic zero field, the singular $k$-chains $C_\ast(\mathscr{C}_m;k)$ are the free $\mathbb{E}_\infty$-$k$-algebra on the singular $k$-chains $C_\ast(BS^1\times (S^1)^m;k)$. The homology of this free $\mathbb{E}_\infty$-$k$-algebra is the free graded commutative algebra on the graded $k$-module $H_\ast(BS^1 \times (S^1)^m;k)$. Since the latter is isomorphic to the free $k$-module generated by $\bigsqcup_{i\geq 0} \Sigma^{2i}\mathscr{S}_\ast$, we conclude that the $E^2$-page and the abutment are isomorphic.
\end{proof}

\begin{corollary} \label{cor:equivariantoverz}
For all $n\geq 0$, the $U(n)$-equivariant homology of $C_2(U(n))$ with integer coefficients is torsion free.
\end{corollary}
\begin{proof}
Consider the spectral sequence of Proposition \ref{prop:ss} for $m=2$ and $k=\Z$,
\[
E^2=\HH^{(S^1)^2}_\ast(\Z[x_0,x_1,x_2,\dots])\, \Longrightarrow\, H_\ast(\mathscr{C}_2;\Z)\, .
\]
By Corollary \ref{cor:hhiterated}, the $E^2$-page is torsion free. Therefore, any non-zero differential would give a non-zero differential after tensoring with $\Q$. This means that the spectral sequence degenerates at the $E^2$-page, because this is what happens over $\Q$ as we noted in Proposition \ref{prop:ss}. It follows that $H_\ast(\mathscr{C}_2;\Z)\cong E^2$, which is therefore torsion free.
\end{proof}

\begin{remark} \label{rem:torsionfreeness}
When $n$ and $m$ are sufficiently large, the equivariant homology of $C_m(U(n))$ contains torsion, and so Corollary \ref{cor:equivariantoverz} does not generalise to arbitrary $m,n$. Indeed, we proved in \cite{Grit25} that for every $m$, the spaces $C_m(U(n))_{hU(n)}$ satisfy homological stability with $\Z$-coefficients as $n$ tends to infinity. The stable homology is the homology of a product of various higher connected covers of $BU$ and $U$; the connectivity of the covers that appear depends on $m$. If $m\geq 7$, then some of these covers have torsion \cite{Singer} and hence so does $C_m(U(n))_{hU(n)}$ for large enough $n$. 
\end{remark}

Under the assumption that $k$ is a field of characteristic zero, let us describe the homology of each path-component of $\mathscr{C}_m$. Choose a maximal torus $T\subseteq U(n)$. Then $T\cong (S^1)^n$ and the Weyl group $S_n$ acts on $T$ by permuting the factors. The inclusion $T\hookrightarrow U(n)$ induces a map
\[
f_n\co BT\times T^m=(T^m)_{hT} \to C_m(U(n))_{hU(n)}\, .
\]
The space $\mathscr{C}_m$ has an $H$-space structure, whose component maps
\[
C_m(U(n))_{hU(n)} \times C_m(U(l))_{hU(l)} \to C_m(U(n+l))_{hU(n+l)}
\]
are induced by block sum of unitary matrices. Using these we can describe the map $f_n$ alternatively as follows: Let $\iota_1\co BS^1 \times (S^1)^m \hookrightarrow \mathscr{C}_m$ be the inclusion of the corresponding disjoint summand (thus, by $BS^1 \times (S^1)^m$ we actually mean $BU(1) \times C_m(U(1))$). Then $f_n$ is the composite of
\[
(\iota_1)^{\times n}\co (BS^1 \times (S^1)^m)^{\times n}\to \mathscr{C}_m^{\times n}
\]
with the multiplication $\mathscr{C}_m^{\times n}\to \mathscr{C}_m$ in the $H$-space structure of $\mathscr{C}_m$.

In fact, the equivalence $\mathscr{C}_m\simeq ku((S^1)^m_+)$ (cf. the proof of Proposition \ref{prop:ss}) and the fact that $ku$ is a special $\Gamma$-space, imply that the $H$-space structure on $\mathscr{C}_m$ extends to an $\mathbb{E}_\infty$-structure. This implies that $f_n$ factors through the homotopy $S_n$-orbits
\[
(BT\times T^m)_{hS_n} \to C_m(U(n))_{hU(n)}\,.
\]
In the proof of Proposition \ref{prop:ss} we noted that $C_\ast(\mathscr{C}_m;k)$ is a free $\mathbb{E}_\infty$-$k$-algebra generated by $C_\ast(BS^1\times (S^1)^m;k)$. This implies that the displayed map induces an isomorphism on homology:
\begin{lemma} \label{lem:char0case}
Let $k$ be a field of characteristic zero. For all $n,m\geq 0$, the inclusion of a maximal torus $T\subseteq U(n)$ induces an isomorphism
\[
H_\ast(BT\times T^m;k)_{S_n}\cong H_{\ast}^{U(n)}(C_m(U(n));k)\, .
\]
\end{lemma}

\begin{remark} \label{rem:sufficesn!}
The conclusion of Lemma \ref{lem:char0case} also follows in an entirely different way from Baird \cite[Theorem 3.5]{Baird}.
\end{remark}

Specialising to the case $m=2$, we obtain the following corollary.

\begin{corollary} \label{cor:embedding}
For every $n\geq 0$, $f_n$ induces an embedding
\[
f_n^\ast\co H^\ast_{U(n)}(C_2(U(n));\Z) \hookrightarrow H^\ast(BT\times T^2;\Z)^{S_n}\,.
\]
\end{corollary}
\begin{proof}
By Lemma \ref{lem:char0case}, the inclusion $T\subseteq U(m)$ induces an isomorphism
\[
H^\ast_{U(n)}(C_2(U(n));\Q)\xrightarrow{\cong} H^\ast(BT\times T^2;\Q)^{S_n}\, .
\]
As $H^\ast_{U(n)}(C_2(U(n));\Z)$ is torsion free by Corollary \ref{cor:equivariantoverz}, the result follows.
\end{proof}

A first description of $H^\ast_{\mathrm{GL}_n(\C)}(C_2(\mathrm{GL}_n(\C));\Z)$ is thus obtained by calculating the image of $f_n^\ast$.

\section{Determination of a generating set} \label{sec:generators}

In this section all (co)homology groups are taken with integer coefficients. Our aim is to compute the image of the embedding $f_n^\ast$ of Corollary \ref{cor:embedding}. Recall that $\mathscr{C}_2$ is the underlying space of an $\mathbb{E}_\infty$-space. Let
\[
\gamma\co \mathscr{C}_2\to \mathscr{C}_2^{gp}
\]
denote its group completion. By the group completion theorem \cite{SegalMcDuff}, $\gamma$ induces an isomorphism
\[
H_\ast(\mathscr{C}_2)[x_0^{-1}] \cong H_\ast(\mathscr{C}_2^{gp})\, ,
\]
where $x_0\in H_0(\mathscr{C}_2)$ is a generator of the monoid of components $\pi_0(\mathscr{C}_2)\cong \N$.

\begin{lemma} \label{lem:gcsplitinjective}
$H_\ast(\mathscr{C}_2)$ is in each degree a finitely generated free $\Z[x_0]$-module.
\end{lemma}
\begin{proof}
We consider the spectral sequence of Proposition \ref{prop:ss}. According to Corollary \ref{cor:hhiterated}, the $E^\infty$-page is given by $E^\infty\cong  \bigotimes_{n\geq 0} \left(\Z[x_n] \otimes {\textstyle \bigwedge_{\Z}}(\sigma x_n,\tau x_n) \otimes \Gamma_\Z(\rho x_n) \right)$, which is a degreewise finitely generated free $\Z[x_0]$-module. The spectral sequence is one of algebras, and in particular one of $\Z[x_0]$-modules. Since the $E^\infty$-page is a free $\Z[x_0]$-module, there are no extension problems and $H_\ast(\mathscr{C}_2)\cong E^\infty$ as $\Z[x_0]$-modules.
\end{proof}

\begin{corollary} \label{cor:gcsurjective}
The induced map $\gamma^\ast\co H^\ast(\mathscr{C}_2^{gp})\to H^\ast(\mathscr{C}_2)$ is surjective.
\end{corollary}
\begin{proof}
By Lemma \ref{lem:gcsplitinjective} and by the group completion theorem the map
\[
\gamma_\ast\co H_\ast(\mathscr{C}_2)\to H_\ast(\mathscr{C}_2^{gp})
\]
is in each degree an inclusion of direct summands of free abelian groups. Hence, its graded dual is surjective.
\end{proof}

We consider the composite
\[
BT\times T^2\xrightarrow{f_n} C_2(U(n))_{hU(n)} \subseteq \mathscr{C}_2 \xrightarrow{\gamma}  \mathscr{C}_2^{gp}\, .
\]
By Corollary \ref{cor:gcsurjective}, the image of $f_n^\ast$ equals the image of $(\gamma f_n)^\ast$. The advantage of using $\gamma f_n$ instead of $f_n$ is the fact that $\mathscr{C}_2^{gp}$ is much easier to describe than $\mathscr{C}_2$. Let $ku$ denote the connective complex $K$-theory spectrum. It follows from the equivalence $\mathscr{C}_2\simeq ku((S^1)^2_+)$, or otherwise from \cite{Lawson}, that there is an equivalence
\begin{equation} \label{eq:einfinityring}
\mathscr{C}_2^{gp} \simeq \Omega^{\infty}(ku\wedge (S^1)^2_+)\, ,
\end{equation}
where $\Omega^\infty$ denotes the underlying space of a spectrum. Here we abuse notation and denote by $ku$ also the connective complex $K$-theory spectrum, and not only the $\Gamma$-space modelling it. The equivalence implies that
\[
\mathscr{C}_2^{gp} \simeq \Z\times BU \times U^2 \times BU
\]
(see the details below). Any choice of polynomial and exterior generators of $H^\ast(BU)$ and $H^\ast(U)$ provide a natural set of generators for $\im((\gamma f_n)^\ast)\cong H_{U(n)}^\ast(C_2(U(n)))$.

Write $f_{n,\infty}=\gamma f_n$. We can express $f_{n,\infty}$, up to homotopy, as the composite
\begin{equation} \label{eq:fm}
BT\times T^2=(BS^1 \times (S^1)^2)^{\times n}\xrightarrow{f_{1,\infty}^{\times n}} (\mathscr{C}_{2,1}^{gp})^{\times n} \to \mathscr{C}_{2,n}^{gp}\, ,
\end{equation}
where $\mathscr{C}_{2,k}^{gp}$ denotes the path-component of $\mathscr{C}_2^{gp}$ indexed by $k\in \Z\cong \pi_0(\mathscr{C}_2^{gp})$, and where the last map is multiplication in the $\mathbb{E}_\infty$-space $\mathscr{C}_2^{gp}$. So this essentially reduces our task to describing the map
\[
f_{1,\infty}^\ast\co H^\ast(\mathscr{C}_{2,1}^{gp})\to H^\ast(BS^1 \times (S^1)^2)\,,
\]
or equivalently, the map $(f_{1,\infty})_\ast\co H_\ast(BS^1\times (S^1)^2)\to H_\ast(\mathscr{C}_{2,1}^{gp})$. To do this, we will use some further structure on $\mathscr{C}_2^{gp}$.

Give $S^1$ its natural abelian group structure. It induces a $ku$-algebra structure on $ku\wedge (S^1)^2_+$. It follows from \cite{Lawson} that the direct sum and tensor product in $U(n)$ define an $\mathbb{E}_\infty$-ring space structure on $\mathscr{C}_2$, and hence on $\mathscr{C}_2^{gp}$, such that (\ref{eq:einfinityring}) is an equivalence of $\mathbb{E}_\infty$-ring spaces. The inclusion $\iota_1\co BS^1\times (S^1)^2 \to \mathscr{C}_2$ is an $H$-map with respect to the multiplicative $\mathbb{E}_\infty$-structure on $\mathscr{C}_2$ (that is, the $\mathbb{E}_\infty$-structure coming from the tensor product). This implies that $f_{1,\infty}$ is an $H$-map as well. Therefore, if we write
\[
(f_{1,\infty})_\ast\co \Gamma_{\Z}[x]\otimes {\textstyle \bigwedge_{\Z}}(a,b)\to H_\ast(\mathscr{C}_{2,1}^{gp})\,,
\]
where $x\in H_2(BS^1)$ and $a,b\in H_1((S^1)^2)$ are generators, it suffices to compute $(f_{1,\infty})_\ast$ on $x,a$ and $b$. Although the image of $(f_{n,\infty})^\ast$ is independent of any specific choices made, we will make various choices of generators along the way; this will eventually provide us with explicit formulas for the images of our chosen generators of $H^\ast(BU)$ and $H^\ast(U)$.

The map $f_{1,\infty}$ can be identified with the map
\[
BS^1 \times (S^1)^2 \cong C_2(U(1))_{hU(1)} \to C_2(U)_{hU}
\]
induced by the inclusion $U(1)\subseteq U$. The map $C_2(U(1))\to C_2(U)$ is a $\pi_1$-isomorphism by \cite[Theorem 1.1]{GPS}. Therefore, $f_{1,\infty}$ is an isomorphism on $H_1$. Pick generators
\begin{equation} \label{eq:basicchoice}
a,b\in H_1(BS^1 \times (S^1)^2)\,.
\end{equation}
Let $a_1,b_1\in H_1(\mathscr{C}_{2,0}^{gp})$ be the classes corresponding to $a$ and $b$, respectively, under the composite isomorphism
\[
\xymatrix{
H_1(BS^1 \times (S^1)^2)\ar[rr]^-{(f_{1,\infty})_\ast}_-{\cong} && H_1(\mathscr{C}_{2,1}^{gp}) \ar[rr]^-{\cdot x_0^{-1}}_-{\cong}  && H_1(\mathscr{C}_{2,0}^{gp})\, .
}
\]
Using the Hurewicz homomorphism
\[
h \co \pi_\ast(\mathscr{C}_{2,0}^{gp})\to H_\ast(\mathscr{C}_{2,0}^{gp})
\]
we define generators $\mu,\nu\in \pi_1(\mathscr{C}_{2,0}^{gp})$ by $a_1=h(\mu)$ and $b_1=h(\nu)$. Let $\beta\in \pi_2(\mathscr{C}_2^{gp})$ be the image of the Bott element under the unit $\Z\times BU\to \mathscr{C}_2^{gp}$.

\begin{lemma}
There is an isomorphism of $\Z[\beta]$-algebras
\[
\pi_\ast(\mathscr{C}_2^{gp})\cong \Z[\beta]\otimes {\textstyle \bigwedge_{\Z}}(\mu,\nu)\,.
\]
\end{lemma}
\begin{proof}
This follows from (\ref{eq:einfinityring}) and the K{\"u}nneth isomorphism.
\end{proof}

Because $\pi_\ast(\mathscr{C}_2^{gp})$ is a free $\Z[\beta]$-module, the generators $\mu,\nu$ and $\mu\nu$ induce a splitting of $\mathbb{E}_\infty$-$(\Z\times BU)$-modules
\begin{equation} \label{eq:homotopytype}
\mathscr{C}_2^{gp}\simeq (\Z\times BU)\times BU\times U^2\, ;
\end{equation}
under it the unit $\Z\times BU\to \mathscr{C}_2^{gp}$ corresponds to the inclusion of the first factor, $\mu,\nu$ correspond to generators of $\pi_1(U^2)$, and $\mu\nu$ corresponds to a generator of $\pi_2(BU)$. By changing the sign of $a$ (and hence of $\mu$) if necessary, we may assume that $\mu\nu$ is the Bott element.

Let $\Z[x_0^{-1},x_n\,|\, n\geq 0]\subseteq H_\ast(\mathscr{C}_2^{gp})$ be the homology ring of $\Z\times BU$, where as before the $x_n$'s are the generators coming from the inclusion $BU(1)\to \{1\}\times BU$. Define the class $q_1=h(\mu\nu)\in H_2(\mathscr{C}_{2,0}^{gp})$. The homology of $\mathscr{C}_2^{gp}$ is a Hopf ring (for the notion of Hopf ring see e.g. \cite{RW}). We let $\circ$ denote the Hopf ring product, that is, the product on $H_\ast(\mathscr{C}_2^{gp})$ induced by the multiplicative $\mathbb{E}_\infty$-structure. For every $n\geq 1$ define
\[
q_n:=x_{n-1}\circ q_1,\quad a_n:=x_{n-1}\circ a_1, \quad b_n:=x_{n-1}\circ b_1\, .
\]
Under (\ref{eq:homotopytype}), the classes $q_1,q_2,\dots$ generate the subspace of primitives of $H_\ast(BU)$, while $a_1,a_2,\dots$ and $b_1,b_2,\dots$ are primitive generators for the exterior algebra $H_\ast(U^2)$.

Let $x\in H_2(BS^1)$ be the generator with $(f_{1,\infty})_\ast(x)=x_1 \in H_{\ast}(\mathscr{C}_{2,1}^{gp})$, and thus with $(f_{1,\infty})_\ast(\gamma_n(x))=x_n$, where $\gamma_n(x)$ is the $n$-th divided power of $x$. Since $f_{1,\infty}$ is an $H$-map, the map induced on homology
\[
(f_{1,\infty})_\ast\co \Gamma_{\Z}[x]\otimes {\textstyle \bigwedge_{\Z}}(a,b)\to H_\ast(\mathscr{C}_{2,1}^{gp})
\]
is a ring homomorphism with respect to the $\circ$-product.

\begin{lemma} \label{lem:imagehomology1}
We have $(f_{1,\infty})_\ast(ab)=x_0 q_1+x_0 a_1 b_1$ in $H_2( \mathscr{C}_{2,1}^{gp})$.
\end{lemma}
\begin{proof}
Using the distributivity law for $\circ$ we get
\[
(f_{1,\infty})_\ast(ab)=(x_0 a_1)\circ (x_0 b_1)=x_0 (a_1\circ b_1)+x_0 a_1b_1\, .
\]
Since $h\co \pi_\ast(\mathscr{C}_{2,0}^{gp})\to H_\ast(\mathscr{C}_{2,0}^{gp})$ is a ring homomorphism with respect to $\circ$, we have that $a_1\circ b_1=h(\mu)\circ h(\nu)=h(\mu\nu)=q_1$.
\end{proof}

Let $c_1,c_2,\dots$ and $d_1,d_2,\dots$ denote the Chern classes of the two $BU$ factors in (\ref{eq:homotopytype}). From now on we use the splitting (\ref{eq:homotopytype}) to identify the cohomology ring of $\mathscr{C}_{2,1}^{gp}$ with that of $BU\times U^2\times BU$. Therefore,
\begin{equation} \label{eq:cohomologyringc2gp}
H^\ast(\mathscr{C}_{2,1}^{gp})= \Z[c_n,d_n\mid n\geq 1]\otimes {\textstyle \bigwedge_\Z}(\alpha_n,\beta_n\mid n\geq 1)\,,
\end{equation}
where the $\alpha_n$ and $\beta_n$ are defined as dual to $x_0 a_n$ and $x_0 b_n$, respectively. Let $t,u,v \in H^\ast(BS^1\times (S^1)^2)$ be generators dual to $x,a,b$, respectively.

\begin{lemma} \label{lem:imagecohomology1}
For every $n\geq 0$, the map $(f_{1,\infty})^\ast\co H^\ast(\mathscr{C}_{2,1}^{gp})\to H^\ast(BS^1\times (S^1)^2)$ sends $\alpha_n\mapsto t^{n-1}u$ (by symmetry, $\beta_n\mapsto t^{n-1}v$) and $d_n\mapsto (-1)^{n-1}t^{n-1}u v$.
\end{lemma}
\begin{proof}
The class $(f_{1,\infty})^\ast(\alpha_n)$ is completely determined by the Kronecker pairings
\[
\lambda_a=\langle(f_{1,\infty})^\ast(\alpha_n),\gamma_{n-1}(x)a\rangle\quad \textnormal{and}\quad \lambda_b=\langle(f_{1,\infty})^\ast(\alpha_n),\gamma_{n-1}(x)b\rangle\, .
\]
To compute $\lambda_a$ we note that $(f_{1,\infty})_\ast(\gamma_{n-1}(x)a)=x_{n-1}\circ (x_0 a_1)$, which by the distributivity law in Hopf rings is $\sum_i x_i a_{n-i}$. Since $\alpha_n$ is dual to $x_0 a_n$, it follows that $\lambda_a=1$. In the same way we see that $\lambda_b=0$, hence $(f_{1,\infty})^\ast(\alpha_n)$ is the class dual to $\gamma_{n-1}(x) a$, which is $t^{n-1}u$.

To compute $(f_{1,\infty})^\ast(d_n)$ we proceed in the same way. This class is determined by the pairings
\[
\lambda_\emptyset=\langle(f_{1,\infty})^\ast(d_n),\gamma_{n}(x)\rangle\quad \textnormal{and}\quad \lambda_{ab}=\langle(f_{1,\infty})^\ast(d_n),\gamma_{n-1}(x)ab\rangle\, .
\]
Since $(f_{1,\infty})_\ast(\gamma_n(x))=x_n$, which is a homology class of the first $BU$-factor, while $d_n$ is a Chern class of the second $BU$-factor, we find that $\lambda_{\emptyset}=0$. On the other hand, by Lemma \ref{lem:imagehomology1} we have
\[
(f_{1,\infty})_\ast(\gamma_{n-1}(x)ab)=x_{n-1}\circ (x_0 q_1+x_0 a_1b_1)=x_0 q_{n} + f\,,
\]
where $f$ is a linear combination of terms $x_i q_{n-i}$ with $i\geq 1$ and terms independent of $q_i$ for any $i$, on all of which $d_n$ evaluates to zero. Hence, $\lambda_{ab}=\langle d_n,x_0 q_n\rangle$. Upon changing path components in $\Z\times BU\times U^2\times BU$, this is the value of the $n$-th Chern class $d_n\in H^{2n}(BU)$ on $q_n\in H_{2n}(BU)$. The Chern character shows that the $n$-th Chern number of the $n$-th power of the Bott element is $(-1)^{n-1}(n-1)!$. Since $h(\beta^{n-1}\mu\nu)=x_1^{\circ (n-1)}\circ q_1=(n-1)! q_n$, 
and $\beta^{n-1}\mu\nu$ is the $n$-th power of the Bott element in $\pi_{2n}(BU)$, we obtain that $\lambda_{ab}=\langle d_n,x_0 q_n\rangle=(-1)^{n-1}$. It follows that $(f_{1,\infty})^\ast(d_n)$ is the class dual to $(-1)^{n-1}\gamma_{n-1}(x)ab$, which is $(-1)^{n-1}t^{n-1}uv$.
\end{proof}

We finally compute the map $(f_{n,\infty})^\ast\co H^\ast( \mathscr{C}_{2,n}^{gp})\to H^\ast(BT\times T^2)$. We identify the cohomology rings $H^\ast(\mathscr{C}_{2,n}^{gp})$ for all $n$ with that in (\ref{eq:cohomologyringc2gp}) (through the equivalences induced by multiplication by $x_0$). Write
\begin{equation} \label{eq:btt2}
H^\ast(BT\times T^2)\cong \Z[t_1,\dots,t_n]\otimes \Lambda(u_1,\dots,u_n)\otimes \Lambda(v_1,\dots,v_n)
\end{equation}
with $|t_i|=2$ and $|u_i|=|v_i|=1$ for all $1\leq i \leq n$. The $n$-fold product $(\mathscr{C}_{2,1}^{gp})^{\times n} \to \mathscr{C}_{2,n}^{gp}$ in the $\mathbb{E}_\infty$-space $\mathscr{C}_2^{gp}$ induces an iterated coproduct
\[
\psi^{(n-1)}\co H^\ast(\mathscr{C}_{2,n}^{gp})\to H^\ast(\mathscr{C}_{2,1}^{gp})^{\otimes n}\, .
\]
To express this coproduct it is convenient to introduce the formal power series $d(s)-1=\sum_{l\geq 1} d_l s^l$, $\alpha(s)=\sum_{l\geq 1} \alpha_l s^{l-1}$ and $\beta(s)=\sum_{l\geq 1} \beta_l s^{l-1}$. Then
\[
\psi^{(n-1)}(d(s))=d(s)^{\otimes n}\textnormal{ and } \psi^{(n-1)}(\alpha(s))=\sum_{i=1}^n 1\otimes \cdots \otimes 1\otimes \alpha(s) \otimes 1\otimes \cdots \otimes 1
\]
where the $i$-th summand has $\alpha(s)$ in the $i$-th spot; similarly, for $\beta(s)$.

\begin{lemma} \label{lem:imagecohomology}
The map $(f_{n,\infty})^\ast\co H^\ast(\mathscr{C}_{2,n}^{gp})\to H^\ast(BT\times T^2)$ satisfies
\[
(f_{n,\infty})^\ast(\alpha(s))=\sum_{i=1}^n \frac{u_i}{1-t_i s},\quad \textnormal{ and }\quad (f_{n,\infty})^\ast(d(s))=\prod_{i=1}^n\left(1+\frac{s u_i v_i}{1+t_i s}\right)\,,
\]
(by symmetry, $(f_{n,\infty})^\ast(\beta(s))=\sum_{i=1}^n \frac{v_i}{1-t_i s}$).
\end{lemma}
\begin{proof}
By (\ref{eq:fm}), the map $(f_{n,\infty})^\ast$ is the composition of $\psi^{(n-1)}$ with $(f_{1,\infty}^\ast)^{\otimes n}$. If we apply $(f_{1,\infty})^\ast$ in the $i$-th tensor factor, then we use the formulas from Lemma \ref{lem:imagecohomology1}, but with variables $t_i,u_i,v_i$ indexed by $i$. In this way we obtain:
\[
(f_{n,\infty})^\ast(\alpha(s))= \sum_{l\geq 1} \sum_{i=1}^n t_i^{l-1}u_i s^{l-1}=\sum_{i=1}^n \frac{u_i}{1-t_i s}\,,
\]
and similarly for $\beta(s)$. We also obtain
\[
(f_{n,\infty})^\ast(d(s))=\prod_{i=1}^n \left(1+\sum_{l\geq 1} (-1)^{l-1} t_i^{l-1}u_i v_i s^l\right)=\prod_{i=1}^n\left(1+\frac{s u_i v_i}{1+t_i s}\right)\,.
\]
\end{proof}

It remains to compute the image of the Chern classes $c_l$. Let $\pi\co \mathscr{C}_2^{gp}\to \Z\times BU$ be the projection. The diagram
\[
\xymatrix{
BT\times T^2\ar[r]^-{f_{n,\infty}} \ar[d]^-{\textnormal{proj.}}&  \mathscr{C}_{2,n}^{gp} \ar[d]^-{\pi} \\
BT\ar[r]^-{\textnormal{incl.}} & \{n\}\times BU
}
\]
commutes up to homotopy. It follows that
\[
(f_{n,\infty})^\ast(c_l)=e_l(t_1,\dots,t_n)\in H^\ast(BT)^{S_n}\subseteq H^\ast(BT\times T^2)
\]
is the $l$-th elementary symmetric polynomial.

For $l\geq 1$, we let $X_l$ and $Y_l$ be the coefficient of $s^{l-1}$ in the expansion of $(f_{n,\infty})^\ast(\alpha(s))$ and $(f_{n,\infty})^\ast(\beta(s))$ as formal power series, respectively, and $Z_l$ the coefficient of $s^l$ in the expansion of $(f_{n,\infty})^\ast(d(s))$. These are expressions in the variables $t_1,\dots,t_n, u_1,\dots,u_n,v_1,\dots,v_n$ which are symmetric under simultaneous permutation of the three sets of variables.

\begin{proposition} \label{prop:3ngenerators}
The image of the embedding $f_n^\ast$ of Corollary \ref{cor:embedding} is generated as a $\Lambda_n$-algebra by $X_l,Y_l$ and $Z_l$ for $l\leq n$.
\end{proposition}
\begin{proof}
It has already been established that the image of $(f_{n,\infty})^\ast$ is generated as a $\Lambda_n$-algebra by $X_{l},Y_{l}$ and $Z_{l}$ for $l\geq 1$. It remains to show that it suffices to take $l\leq n$. To this end, consider the identity
\[
\left(\sum_{k=0}^n (-1)^k e_k s^k\right)\left(\sum_{l\geq 1} X_l s^{l-1}\right)=\left(\prod_{i=1}^n(1-t_is)\right)\left(\sum_{i=1}^n \frac{u_i}{1-t_i s}\right)\, .
\]
The right hand side is a polynomial in $s$ of degree $\leq n-1$. Hence, on the left hand side the coefficient of $s^l$ for each $l\geq n$ must vanish. But for $l\geq n$ this coefficient is $\sum_{k=0}^n (-1)^k e_k X_{l+1-k}$. This recurrence shows that all $X_l$ for $l> n$ lie in the $\Lambda_n$-linear span of the $X_l$ for $l\leq n$. The same argument applies to the $Y_l$. For $Z_l$ consider likewise the identity
\[
\left(\sum_{k=0}^n e_k s^k\right)\left(1+\sum_{l\geq 1} Z_l s^l\right)=\left(\prod_{i=1}^n (1+t_i s)\right) \prod_{i=1}^n\left(1+\frac{s u_i v_i}{1+t_i s}\right)\, .
\]
The right hand side is a polynomial in $s$ of degree $\leq n$. On the left hand side this implies the recurrence $\sum_{k=0}^n e_k Z_{l-k}$ for all $l > n$.
\end{proof}

Expanding the power series of Lemma \ref{lem:imagecohomology} leads to the following explicit formulae:
\begin{equation} \label{eq:explicitformxy}
X_l=\sum_{i=1}^n t_i^{l-1} u_i\,,\quad (\textnormal{by symmetry, }Y_l=\sum_{i=1}^n t_i^{l-1} v_i)\,.
\end{equation}
The formula for the $Z_l$ is a bit more involved.

\begin{definition}
For integers $1\leq k_1 < \cdots < k_j \leq n$ the \emph{complete homogeneous symmetric polynomial} of degree $q\geq 0$ in the variables $t_{k_1},\dots,t_{k_j}$ is defined by
\[
h_q(t_{k_1},\dots,t_{k_j})=\sum_{\substack{(i_1,\cdots , i_j)\in \N^j, \\ i_1+\dots+i_j = q }}t_{k_1}^{i_1}\cdots t_{k_j}^{i_j}\, .
\]
\end{definition}

Using the fact that $(u_iv_i)^2=0$ for all $1\leq i \leq n$, we get after a short computation
\begin{equation} \label{eq:explicitformz}
Z_l=\sum_{k=0}^{\min(l,n)} (-1)^{l-k} \sum_{\substack{I\subseteq \{1,\dots,n\}\\ |I|=k}} h_{l-k}(t_I) \prod_{j\in I} u_j v_j\, ,
\end{equation}
where $t_I=\{t_j\}_{j \in I}$.

\section{Determination of the ideal of relations} \label{sec:ideal}

We shall now derive the relations amongst the generators $X_1,\dots,X_n,Y_1,\dots,Y_n$ and $Z_1,\dots, Z_n$ from Section \ref{sec:generators}. Let
\[
B=H^\ast(BT\times T^2)=\Z[t_1,\dots,t_n]\otimes {\textstyle \bigwedge_\Z}(u_1,\dots,u_n,v_1,\dots,v_n)\, .
\]
Recall from (\ref{eq:explicitformxy}) that $X_l=\sum_{i=1}^n t_{i}^{l-1} u_i$ and $Y_l=\sum_{i=1}^n t_{i}^{l-1} v_i$. For $1\leq l\leq n$ this can be written succinctly, using the Vandermonde matrix $V=(t_j^{i-1})_{1\leq i,j\leq  n}$, as
\[
(X_1,\dots,X_n)^T=V(u_1,\dots,u_n)^T \quad\textnormal{and}\quad (Y_1,\dots,Y_n)^T=V(v_1,\dots,v_n)^T \, .
\]
If $\Delta$ is the determinant of $V$ and $\mathrm{adj}(V)=\Delta V^{-1}$ the adjugate of $V$, then we can rewrite the first identity as $\Delta(u_1,\dots,u_n)^T=\mathrm{adj}(V) (X_1,\dots,X_n)^T$, and the second identity accordingly.

For an $n\times n$ matrix $E$ and $k$-element subsets $I,J\subseteq \{1,\dots,n\}$ we denote by $E_{I,J}$ the $k\times k$ submatrix of $E$ consisting of the rows and columns of $E$ whose indices are in $I$ and $J$, respectively.

\begin{lemma} \label{lem:adjugate}
Let $K\subseteq \{1,\dots,n\}$ be a $k$-element subset. In $B$ we have that
\[
\Delta^{2k}\prod_{j\in K} u_j v_j=(-1)^{\frac{k(k-1)}{2}}\sum_{\substack{I,J\subseteq \{1,\dots,n\}\\ |I|=|J|=k}} \det(\mathrm{adj}(V)_{I,K}^T \mathrm{adj}(V)_{K,J}) X_I Y_J\, .
\]
Moreover, the coefficient of $X_I Y_J$ on the right hand side is divisible by $\Delta^{2k-2}$.
\end{lemma}
\begin{proof}
For ease of notation, we write $A=\mathrm{adj}(V)$. Now we get
\begin{equation*}
\begin{split}
\Delta^{2k} \prod_{j\in K} u_j v_j&=\prod_{j\in K} \left(\sum_{i=1}^n A_{j,i} X_i\right) \left(\sum_{i=1}^n A_{j,i} Y_i\right) \\&=(-1)^{\frac{k(k-1)}{2}} \prod_{j\in K} \left(\sum_{i=1}^n A_{j,i} X_i\right) \prod_{j\in K}\left(\sum_{i=1}^n A_{j,i} Y_i\right)\,,
\end{split}
\end{equation*}
but notice that $\prod_{j\in K} \left(\sum_{i=1}^n A_{j,i} X_i\right)= \sum_{\substack{I\subseteq \{1,\dots,n\} \\ |I|=k}} \det(A_{K,I})X_I$, and similarly for the second product. This leads to the claimed formula.

Divisibility by $\Delta^{2k-2}$ follows from Jacobi's lemma on complementary minors (see e.g. \cite[\S 42]{Aitken}) which says, in particular, that a $k\times k$ minor of the adjugate of a matrix $E$ is divisible by $\det(E)^{k-1}$.
\end{proof}

Define $D(s)=\mathrm{diag}(s/(1+t_1 s),\dots,s/(1+t_n s))$ and $M(s)=\mathrm{adj}(V)^T D(s) \mathrm{adj}(V)$.

\begin{lemma} \label{lem:relations}
In $B$ the following identity holds:
\[
\Delta^2\sum_{l\geq 0} Z_l s^l=\sum_{k=0}^n (-1)^{\frac{k(k-1)}{2}} \sum_{\substack{I,J\subseteq \{1,\dots,n\}\\|I|=|J|=k}} \det(M(s)_{I,J})/\Delta^{2k-2} X_I Y_J\, .
\]
\end{lemma}
\begin{proof}
Let us write $A=\mathrm{adj}(V)$. By Lemma \ref{lem:imagecohomology} we have
\begin{equation*}
\begin{split}
\Delta^{2k} \sum_{l\geq 0} Z_l s^l& =\Delta^{2k} \prod_{i=1}^n\left(1+\frac{s u_i v_i}{1+t_i s}\right)\\&=\sum_{k=0}^n \sum_{\substack{K\subseteq \{1,\dots,n\} \\ |K|=k}}  \left(\Delta^{2k} \prod_{j\in K} u_j v_j \right)\underbrace{\prod_{j\in K} \left(\frac{s}{1+t_j s}\right)}_{=\det(D(s)_{K,K})}\, ,
\end{split}
\end{equation*}
and using Lemma \ref{lem:adjugate} this becomes
\[
 \sum_{k=0}^n (-1)^{\frac{k(k-1)}{2}} \sum_{\substack{I,J\subseteq \{1,\dots,n\}\\ |I|=|J|=k}}\underbrace{\left( \sum_{\substack{K\subseteq \{1,\dots,n\} \\ |K|=k}}  \det(A_{I,K}^T D(s)_{K,K}A_{K,J}) \right)}_{=\det(M(s)_{I,J})}X_I Y_J\, .
\]
Let us explain why the term in parenthesis is equal to $\det(M(s)_{I,J})$. It follows from the Cauchy-Binet formula that if $E$ and $F$ are $n\times n$ matrices and $I,J\subseteq \{1,\dots,n\}$ are $k$-element subsets, then $\det((EF)_{I,J})=\sum_{\substack{K\subseteq \{1,\dots,n\}\\ |K|=k}}\det(E_{I,K}F_{K,J})$. Since $D(s)$ is diagonal, this implies that
\[
\det(D(s)_{K,K}A_{K,J})=\det((D(s)A)_{K,J})\,.
\]
Applied once more to $E=A^T$ and $F=D(s)A$ proves the claim.

Finally, dividing by $\Delta^{2k-2}$ yields the assertion of the lemma.
\end{proof}

The relations $R_1,\dots,R_n$ from the introduction (\ref{eq:idealofrelations}) now follow by comparing coefficients of $s^l$ for $1\leq l \leq n$ in the identity of Lemma \ref{lem:relations}.

Let us now treat $X_1,\dots,X_n,Y_1,\dots,Y_n,Z_1,\dots,Z_n$ as formal variables of degree $|X_i|=|Y_i|=2i-1$ and $|Z_i|=2i$. Let
\[
A:=\Lambda_n\otimes \Z[Z_1\dots,Z_n]\otimes {\textstyle \bigwedge_\Z}(X_1,\dots,X_n,Y_1,\dots,Y_n)\, .
\]
and let $\rho\co A\to B$ be the map of $\Lambda_n$-algebras which substitutes for the generators the explicit formulae of (\ref{eq:explicitformxy}) and (\ref{eq:explicitformz}).

\begin{lemma} \label{lem:ideal}
The kernel of $\rho$ is the $\Delta^2$-saturated ideal $\mathcal{J}$ of Theorem \ref{thm:main}.
\end{lemma}
\begin{proof}
Consider the commutative diagram
\[
\xymatrix{
A\ar[r] \ar[d]^-{\rho} & A[\Delta^{-2}] \ar[d] \ar[r] & A[\Delta^{-2}]/(R_1,\dots,R_n) \ar[d]^-{\bar{\rho}} \\
B\ar[r] & B[\Delta^{-1}] \ar[r]^-{\cong} & \Z[t_1,\dots,t_n][\Delta^{-1}] \otimes {\textstyle \bigwedge_{\Z}}(X_1,\dots,X_n,Y_1,\dots,Y_n)
}
\]
The two horizontal arrows on the left are the canonical maps into the localisation, both of which are injective. Over the ring $B[\Delta^{-1}]$ the Vandermonde matrix $V$ is invertible. The isomorphism at the bottom is obtained by the transformation $(u_1,\dots,u_n)^T\mapsto V^{-1} (X_1,\dots,X_n)^T$ and $(v_1,\dots,v_n)^T\mapsto V^{-1} (Y_1,\dots,Y_n)^T$.

This makes the composition from $A$ into the lower right corner of the diagram the identity on $X_1,\dots,X_n$ and $Y_1,\dots,Y_n$. By Lemma \ref{lem:relations}, this composite factors through the quotient by $(R_1,\dots,R_n)$. But with $\Delta^2$ inverted, this quotient kills precisely the generators $Z_1,\dots,Z_n$ and nothing more.

The induced map $\bar{\rho}$ is the identity on the exterior part and the canonical inclusion $\Lambda_n[\Delta^{-2}]\to \Z[t_1,\dots,t_n][\Delta^{-1}]$ on the coefficients; in particular, it is injective. It follows that the kernel of $\rho$ is the preimage of $(R_1,\dots,R_n)$ under the localisation $A\to A[\Delta^{-2}]$, which is precisely the $\Delta^2$-saturation of $(R_1,\dots,R_n)$.
\end{proof}

This concludes the proof of the first part of Theorem \ref{thm:main}.

\section{Minimality of the generating set} \label{sec:minimality}

We prove here the second statement of Theorem \ref{thm:main}.

\begin{proposition}
Let $\mathscr{G}=\{P_1,\dots,P_k\}\sqcup \{Q_1,\dots,Q_l\}$ be a finite homogeneous generating set of $A/\mathcal{J}$ as a $\Lambda_n$-algebra, where the $P_i$ have odd degree and the $Q_i$ have even degree. Then $k \geq 2n$ and $l \geq n$.
\end{proposition}
\begin{proof}
We think of $A/\mathcal{J}$ as a $\Lambda_n$-subalgebra of $B$, more precisely of the invariants $B^{S_n}$ (notation as in Section \ref{sec:ideal}). We will first show that $k\geq 2n$. Since we may always modify the elements of $\mathscr{G}$ by adding elements of $\Lambda_n$, we may assume without loss of generality that
\[
\mathscr{G} \subseteq (u_1,\dots,u_n,v_1,\dots,v_n)\subseteq B\, .
\]
Let $\pi\co B\to B/(u_i u_j,u_i v_j,v_iv_j \mid 1\leq i,j\leq n)$ be the projection onto the free $\Z[t_1,\dots,t_n]$-module generated by $\{u_1,\dots,u_n,v_1,\dots,v_n\}$. We have that $\pi(Q_j)=0$ for all $1\leq j \leq l$, since the $Q_j$, having even degree, must be at least quadratic in the variables $\{u_1,\dots,u_n,v_1,\dots,v_n\}$. We also have that $\pi(X_{i})=X_{i}$ and $\pi(Y_{i})=Y_{i}$ for all $1\leq i \leq n$. Since $\mathscr{G}$ generates $A/\mathcal{J}$, it follows that the $X_{i}$ and $Y_{i}$ for $1\leq i \leq n$ lie in the $\Z[t_1,\dots,t_n]$-linear span of $\{\pi(P_1),\dots,\pi(P_k)\}$. Now forget the grading and consider base change to the field of rational functions $\Q(t_1,\dots,t_n)$. Since the Vandermonde matrix becomes invertible, the $\{u_1,\dots,u_n\}$ and $\{v_1,\dots,v_n\}$ lie in the $\Q(t_1,\dots,t_n)$-linear span of $\{X_{1},\dots,X_{n}\}$ and $\{Y_{1},\dots,Y_{n}\}$. This implies that the $\Q(t_1,\dots,t_n)$-linear span of $\{\pi(P_1),\dots,\pi(P_k)\}$ contains the $\Q(t_1,\dots,t_n)$-linear span of $\{u_1,\dots,u_n,v_1,\dots,v_n\}$. For dimension reasons it follows that $k\geq 2n$.

Next we show that $l\geq n$. Consider the ring homomorphism
\[
\tau \co B\to \bar{B}:={\textstyle \bigwedge_\Z}(u_1,\dots,u_n,v_1,\dots,v_n)
\]
defined by sending $t_i\mapsto 0$ for all $1\leq i \leq n$. Let $\bar{A}:=\tau(A/\mathcal{J})$ and write $\bar{a}:=\tau(a)$ whenever $a\in A/\mathcal{J}$. Then $\bar{X}_{1}=\sum_i u_i$, $\bar{Y}_{1}=\sum_{i} v_i$, and $\bar{X}_{i}=\bar{Y}_{i}=0$ for all $2\leq i \leq n$.
Furthermore, (\ref{eq:explicitformz}) implies $\bar{Z}_{i}=e_i(u_1v_1,\dots,u_nv_n)$ for all $1\leq i \leq n$, where $e_i$ is the $i$-th elementary symmetric polynomial. In $\bar{A}\subseteq \bar{B}$ the following divided power relations hold:
\[
\bar{Z}_{i} \bar{Z}_{j}=\begin{cases} \binom{i+j}{i} \bar{Z}_{i+j} & \textnormal{if } i+j \leq n,\\ 0 & \textnormal{otherwise.} \end{cases}
\]
Let $A'\subseteq \bar{A}$ be the subring generated by $\bar{Z}_{1},\dots,\bar{Z}_{n}$. Direct calculation shows that $\bar{X}_{1}\bar{Y}_{1} \bar{Z}_{i}$ cannot be a non-zero multiple of $\bar{Z}_{i+1}$; this implies a splitting of graded abelian groups $
\bar{A}\cong A'\oplus (\bar{X}_{1},\bar{Y}_{1}) A'$, and thus an isomorphism
\[
\bar{A}/(\bar{X}_{1},\bar{Y}_{1}) \cong A'=\Gamma_{\Z}[\bar{Z}_{1}]/(\gamma_{n+1}(\bar{Z}_{1}),\gamma_{n+2}(\bar{Z}_{1}),\dots)\, .
\]
Over $\Z$ this algebra is generated by no fewer than $n$ elements.  The $P_i \in \mathscr{G}$, having odd degree, are mapped to zero in $A'$. Thus, $A'$ is generated by $\{\bar{Q}_1,\dots,\bar{Q}_l\}$, which implies that $l\geq n$.
\end{proof}

The proof of Theorem \ref{thm:main} is now finished.

As an addendum we show that if $n$ is inverted in the coefficient ring, then the size of a minimal generating set can be reduced by one. Let $x_1,x_2,x_3,\dots$ be variables. Set $x_0:=1$ and define the formal power series $f(s)=\sum_{k\geq 0} (-1)^k x_k s^k$. We define the $k$-th Newton polynomial $s_k(x_1,\dots,x_k)$ by the identity
\[
\sum_{k\geq 1} s_k(x_1,\dots,x_k) s^{k-1} = -\frac{f'(s)}{f(s)}\, .
\]
For example, $s_1(x_1)=x_1$, $s_2(x_1,x_2)=x_1^2-2x_2$, $s_3(x_1,x_2,x_3)=x_1^3-3x_1x_2+3x_3$ and so on. Continue to denote $\Z[t_1,\dots,t_n]\otimes \bigwedge_\Z(u_1,\dots,u_n)\otimes \bigwedge_\Z(v_1,\dots,v_n)$ by $B$.

\begin{lemma} \label{lem:wl}
For every $k\geq 1$, the following identity holds in $B$:
\[
s_k(Z_1,\dots,Z_k)=k\sum_{i=1}^n t_i^{k-1} u_i v_i\, .
\]
\end{lemma}
\begin{proof}
Define the formal power series $Z(s)=\sum_{k\geq 0} Z_k s^k$, where $Z_0:=1$. Then
\[
-\log(Z(-s))'=\sum_{k\geq 1} s_k(Z_1,\dots,Z_k) s^{k-1}\, .
\]
By Lemma \ref{lem:imagecohomology} we have that $Z(-s)=\prod_{i=1}^n \left(1-\frac{s u_i v_i}{1-t_i s}\right)$, and so we get
\[
-\log(Z(-s))'=-\frac{\mathrm{d}}{\mathrm{d}s}\sum_{i=1}^n \left(1-\frac{s u_i v_i}{1-t_i s}\right)=\sum_{i=1}^n \frac{u_i v_i}{(1-t_i s)^2}\, .
\]
Now expand it as a power series in $s$. Then the coefficient of $s^{k-1}$ is precisely $k\sum_{i=1}^n t_i^{k-1} u_i v_i$.
\end{proof}

\begin{lemma} \label{lem:newtonidentity}
In $A/\mathcal{J}$ the following relation holds:
\[
\sum_{i=0}^{n-1} (-1)^{i+1}e_i\left(s_{n-i}(Z_1,\dots,Z_{n-i})-\sum_{j=0}^{n-i-1} X_{n-i-j} Y_{j+1}\right)=0\, .
\]
\end{lemma}
\begin{proof}
Consider the generating functions
\[
X(s)=\sum_{i=1}^n \frac{u_i}{1-t_i s},\quad Y(s)=\sum_{i=1}^n \frac{v_i}{1-t_i v_i},\quad S(s)=\sum_{i=1}^n \frac{u_i v_i}{(1-t_i s)^2}\, .
\]
Then
\[
X(s)Y(s)-S(s)=\sum_{\substack{i,j=1\\ i\neq j}}^n \frac{u_i v_j}{(1-t_i s)(1-t_j s)}\,,
\]
and thus, letting $E(s)=\prod_{i=1}^n (1-t_i s)$ be the generating function for the elementary symmetric polynomials, the expression $E(s)(X(s)Y(s)-S(s))$ is a polynomial in $s$ of degree at most $n-2$. Then the coefficient of $s^{l}$ for all $l\geq n-1$ must vanish. But plugging in $X(s)=\sum_{l\geq 1} X_l s^{l-1}$, $Y(s)=\sum_{l\geq 1} Y_l s^{l-1}$ and $S(s)=\sum_{l\geq 1} s_l(Z_1,\dots,Z_l) s^{l-1}$ (see Lemma \ref{lem:imagecohomology} and the proof of Lemma \ref{lem:wl}), the coefficient of $s^{n-1}$ is precisely the expression which the lemma claims to vanish.
\end{proof}

\begin{addendum} \label{add:minimal}
For every $n\geq 1$, $H^\ast_{\mathrm{GL}_n(\C)}(C_2(\mathrm{GL}_n(\C));\Z[1/n])$ is generated as a $\Lambda_n\otimes \Z[1/n]$-algebra by $3n-1$ elements. It cannot be generated by fewer elements.
\end{addendum}
\begin{proof}
In the identity of Lemma \ref{lem:newtonidentity}, $Z_{n}$ appears only once, namely in $s_n(Z_{1},\dots,Z_{n})$ with coefficient $n$. The relation shows that if $n$ is inverted, then the generator $Z_n$ can be omitted. In \cite[Theorem F]{KT21} Kishimoto--Takeda showed that any generating set for the rational cohomology ring of $C_2(\mathrm{GL}_n(\C))$ has to have at least $3n-1$ elements. By Proposition \ref{prop:equivariantlyformal} below, the natural map from the rational equivariant to the rational non-equivariant cohomology of $C_2(\mathrm{GL}_n(\C))$ is surjective. This implies that a generating set for the equivariant cohomology as a $\Lambda_n\otimes \Z[1/n]$-algebra must have at least $3n-1$ elements as well. 
\end{proof}

\section{Inverting $n!$} \label{sec:inverting}

The goal of this section is to determine a presentation of the non-equivariant cohomology ring  $H^\ast(C_2(\mathrm{GL}_n(\C));k)$, where $k$ is a field of characteristic zero or prime to $n!$. If the integers $1,\dots,n$, equivalently $n!$, are inverted, we may define elements
\[
W_l=s_l(Z_1,\dots,Z_l)/l
\]
in $A/\mathcal{J} \otimes \Z[1/n!]$. These can be taken as generators instead of the $Z_1,\dots,Z_n$, because $s_l(Z_1,\dots,Z_l)$ always has leading term $(-1)^{l+1} l Z_l$. In this basis the ideal of relations simplifies considerably: Write $X^T=(X_1,\dots,X_n)$ etc., and let $A_l$ denote the symmetric matrix
\[
A_l=\mathrm{adj}(V)^T \mathrm{diag}(t_1,\dots,t_n)^{l-1} \mathrm{adj}(V)\in \mathrm{Mat}_n(\Z[e_1,\dots,e_n])\,.
\]
Let $\mathcal{J}\subseteq A$ be as in Theorem \ref{thm:main}.

\begin{proposition} \label{prop:relationswhenn!inverted}
 After inverting $n!$, the ideal $\mathcal{J}\subseteq A$ is the $\Delta^2$-saturation of the ideal generated by the elements $R_l':=\Delta^2 W_l-X^T A_l Y$ for $1\leq l \leq n$.
\end{proposition}
\begin{proof}
Let $u=(u_1,\dots,u_n)^T$  and $v=(v_1,\dots,v_n)^T$. In the ring $B$ we have that $W_l=\sum_{k=1}^n t_k^{l-1}u_k v_k$ by Lemma \ref{lem:wl}. This implies that
\begin{equation*}
\begin{split}
\Delta^2 W_l &=(\Delta u)^T\mathrm{diag}(t_1,\dots,t_n)^{l-1} (\Delta v)\\&= (\mathrm{adj}(V) X)^T\mathrm{diag}(t_1,\dots,t_n)^{l-1} (\mathrm{adj}(V) Y)\\&= X^TA_l Y\, .
\end{split}
\end{equation*}
The rest of the proof is the same as in Lemma \ref{lem:ideal}.
\end{proof}

After inverting $n!$, Lemma \ref{lem:newtonidentity} allows us to eliminate the generator $W_n$. As the proof of Lemma \ref{lem:ideal} shows, the relation $R_n'$ of Proposition \ref{prop:relationswhenn!inverted} must then be a consequence of $R_1',\dots,R_{n-1}'$.

We also need the following equivariant formality result:

\begin{proposition} \label{prop:equivariantlyformal}
Let $k$ be a field of characteristic zero or prime to $n!$. Then, for all $m\geq 0$, $H^\ast_{U(n)}(C_m(U(n));k)$ is a free $H^\ast(BU(n);k)$-module.
\end{proposition}

This follows as an application of \cite[Proposition A.3]{Baird} (which uses the Borel localisation theorem), but we give here a more direct proof. Our proof (more precisely, line (\ref{eq:comodule})) could be used to give a formula for the Poincar{\'e} series of $C_m(U(n))$; but because such formulae are well-known (see e.g. \cite{RamrasStafa}), we will not do this here.

Before we embark on the proof we set some notation. View $C:=H_\ast(BS^1;k)$ as a coalgebra via the diagonal map $BS^1\to BS^1\times BS^1$. The graded dual $C^\vee$ is the polynomial algebra $k[t]$ in a single variable of degree $2$. For $a\geq 0$, let
\[
S^a(C)=(C^{\otimes a})_{S_a}
\]
be the $a$-th symmetric power of $C$ with the coalgebra structure induced from $C$. Its graded dual is the algebra of symmetric polynomials $S^a(C)^\vee=k[t_1,\dots,t_a]^{S_a}$.

For $a\geq 0$, we let ${\textstyle \bigwedge}^a(C)$ denote the $a$-th exterior power of $C$. Its graded dual is the subspace ${\textstyle \bigwedge}^a(C)^\vee \subseteq k[t_1,\dots,t_a]$ of alternating polynomials. This is naturally a $S^a(C)^\vee$-module, which shows that ${\textstyle \bigwedge}^a(C)$ is naturally a $S^a(C)$-comodule.

\begin{itemize}
\item For any $a,b\geq 0$, we can view $S^{a}(C) \otimes S^{b}(C)$ as a $S^{a+b}(C)$-comodule via the quotient map $S^{a}(C) \otimes S^{b}(C)\to S^{a+b}(C)$.  This is a free $S^{a+b}(C)$-comodule; indeed, the dual statement says that $k[t_1,\dots,t_{a+b}]^{S_{a}\times S_{b}}$ is a free $k[t_1,\dots,t_{a+b}]^{S_{a+b}}$-module, which is a classical fact from invariant theory.
\item  The $S^a(C)$-comodule ${\textstyle \bigwedge}^a(C)$ is also free: The dual statement says that the subring of $k[t_1,\dots,t_a]$ of alternating polynomials is a free module over the ring of symmetric polynomials; this is indeed so, the free generator being the Vandermonde polynomial.
\end{itemize}

\begin{proof}[Proof of Proposition \ref{prop:equivariantlyformal}]
We find it easier to formulate the proof of the dual -- and equivalent -- statement that $H_\ast^{U(n)}(C_m(U(n));k)$ is a free $H_\ast(BU(n);k)$-comodule. Choose a maximal torus $T\subseteq U(n)$ and identify $H_\ast(BU(n);k)$ with $H_\ast(BT;k)_{S_n}$. The inclusion $T\subseteq U(n)$ induces an isomorphism of $H_\ast(BU(n);k)$-comodules
\[
H_\ast^{U(n)}(C_m(U(n));k)\cong H_\ast(BT\times T^m;k)_{S_n}\, .
\]
For $\mathrm{char}(k)=0$ a self-contained proof was given in Lemma \ref{lem:char0case}. If $\mathrm{char}(k)$ is prime to $n!$, then the same conclusion holds by \cite[Theorem 3.5]{Baird}, taking equivariant cohomology and arguing using the Serre spectral sequence.

We will now simply calculate the right hand side as a $H_\ast(BT;k)_{S_n}$-comodule and observe that it is free. First, observe there is an isomorphism of coalgebras $H_\ast(BT;k)_{S_n}\cong S^n(C)$. Let $\mathscr{B}$ be a graded basis for $H_\ast((S^1)^m;k)$ as a $k$-module. Let $\mathscr{B}^{\textnormal{odd}} \subseteq \mathscr{B}$ be the subset of odd degree and $\mathscr{B}^{\textnormal{even}}\subseteq \mathscr{B}$ the subset of even degree elements, respectively. Let $\mathcal{P}(\mathscr{B},m)$ be the set of $\mathscr{B}$-indexed partitions $\sum_{x\in \mathscr{B}} \lambda_x=m$, where $\lambda_x\geq 0$ for all $x\in \mathscr{B}$. Note that
\[
H_\ast(BT\times T^m;k)_{S_n}\cong (H_\ast(BS^1\times (S^1)^m)^{\otimes n})_{S_n}\, .
\]
By simply writing down a basis in terms of $\mathscr{B}$ one sees that this is isomorphic as a $S^n(C)$-comodule to
\begin{equation} \label{eq:comodule}
\bigoplus_{\lambda\in \mathcal{P}(\mathscr{B},m)} \left(\bigotimes_{x\in \mathscr{B}^{\textnormal{odd}}} {\textstyle \bigwedge}^{\lambda_x}(\Sigma^{|x|} C)\right)\otimes \left(\bigotimes_{x\in \mathscr{B}^{\textnormal{even}}}S^{\lambda_x}(\Sigma^{|x|}C)\right)\, .
\end{equation}
Here $\Sigma^{|x|}$ means a shift in grading by $|x|$. By the remarks preceding the proof, this is a free $S^n(C)$-comodule.
\end{proof}

As a consequence of the proposition (using, for example, the Eilenberg-Moore spectral sequence), the natural map
\[
k\otimes_{H^\ast(BU(n);k)}H^\ast_{U(n)}(C_m(U(n));k) \to H^\ast(C_m(U(n));k)
\]
is an isomorphism of $k$-algebras. This means that, after change of coefficients to $k$, a presentation of $H^\ast(C_2(U(n));k)$ is obtained from the presentation $A/\mathcal{J}$ of Theorem \ref{thm:main} by reducing modulo the irrelevant ideal $(e_1,\dots,e_n)\subseteq \Lambda_n$.

We now put everything together: Let $k$ be a field of characteristic zero or prime to $n!$ and let $\Lambda_n=k[e_1,\dots,e_n]$. Omitting the generator $W_n$, let $A'$ be the free graded commutative algebra over $\Lambda_n$ generated by $X_1,\dots,X_n$, $Y_1,\dots,Y_n$ and $W_1,\dots,W_{n-1}$. Let $\mathcal{J}'\subseteq A'$ be the $\Delta^2$-saturation of the ideal $(R_1',\dots,R_{n-1}')$. So, $A'/\mathcal{J}'\cong A/\mathcal{J}\otimes k$.

Let $\bar{X}_1,\dots,\bar{X}_n$, $\bar{Y}_1,\dots,\bar{Y}_n$ and $\bar{W}_1,\dots,\bar{W}_{n-1}$ be generators of degrees $|\bar{W}_l|=2l$ and $|\bar{X}_l|=|\bar{Y}_l|=2l-1$. We think of these as the reduction modulo $(e_1,\dots,e_n)\subseteq \Lambda_n$ of the generators $X_l,Y_l$ and $W_l$ of $A'$. Let $\bar{A}'$ be the $k$-algebra
\[
\bar{A}'=k[\bar{W}_1,\dots,\bar{W}_{n-1}] \otimes {\textstyle \bigwedge_k}(\bar{X}_1,\dots,\bar{X}_n,\bar{Y}_1,\dots,\bar{Y}_n)\, .
\]
Let $\bar{\mathcal{J}}'\subseteq \bar{A}'$ be the reduction of $\mathcal{J}'$ modulo $(e_1,\dots,e_n)$.

\begin{theorem} \label{thm:main2}
Let $k$ be a field of characteristic zero or prime to $n!$. For every $n\geq 0$, there is an isomorphism of $k$-algebras
\[
H^\ast(C_2(\mathrm{GL}_n(\C));k) \cong \bar{A}'/\bar{\mathcal{J}}'\, .
\]
\end{theorem}

We shall now explain a simple procedure by which one can compute a $k$-linear basis for the ideal of relations $\bar{\mathcal{J}}'$ (see Corollary \ref{cor:main2}). To this end, we introduce a bigrading on $A'/\mathcal{J}'$. To distinguish it from the cohomological grading, to which it is not related at all, we shall call it a \emph{weight}. The generators shall have weights
\[
\mathsf{w}(e_l)=(0,0),\quad \mathsf{w}(X_l)=(1,0),\quad \mathsf{w}(Y_l)=(0,1),\quad \mathsf{w}(W_l)=(1,1)\,,
\] 
independent of $l$. The relations $R_l'$ of Proposition \ref{prop:relationswhenn!inverted} are then homogeneous of weight $(1,1)$. It follows that the ideal $\mathcal{J}'$ is homogeneous with respect to weight. Hence, $A'/\mathcal{J}'$ inherits a bigrading, and so does $\bar{A}'/\bar{\mathcal{J}}'$ with $\mathsf{w}(\bar{X}_l)=(1,0)$, $\mathsf{w}(\bar{Y}_l)=(0,1)$ and $\mathsf{w}(\bar{W}_l)=(1,1)$.

\begin{lemma} \label{lem:monomials}
Every monomial in the generators $X_l,Y_l$ and $W_l$ of total weight $(a,b)$ with $a> n$ or $b> n$ is zero in $A'/\mathcal{J}'$.
\end{lemma}
\begin{proof}
We can view $A'/\mathcal{J}'$ as a subring of $B\otimes k$. Then, by (\ref{eq:explicitformxy}) and Lemma \ref{lem:wl}, ``weight $(a,b)$" translates into ``homogeneous of degree $a$ in the variables $u_1,\dots,u_n$ and homogeneous of degree $b$ in the variables $v_1,\dots,v_n$". Since these are generators of the exterior algebra part of $B\otimes k$, any monomial of degree $> n$ in either one of those sets of variables is zero.
\end{proof}

We can now construct the quotient $\bar{A}'/\bar{\mathcal{J}}'$ in two steps: Let $\mathcal{M}\subseteq A'$ be the ideal generated by all monomials of total weight $(a,b)$ with $a> n$ or $b> n$. Lemma \ref{lem:monomials} shows that $\mathcal{M}\subseteq \mathcal{J}'$, and thus $A'/\mathcal{J}'\cong (A'/\mathcal{M})/(\mathcal{J}'/\mathcal{M})$. Let $\pi\co A'\to A'/\mathcal{M}$ be the projection, and let $\mathcal{I}\subseteq A'/\mathcal{M}$ be the $\Delta^2$-saturation of $(\pi(R_1'),\dots,\pi(R_{n-1}'))$. Let $\bar{\mathcal{I}}$ be its reduction modulo $(e_1,\dots,e_n)$. Then $\mathcal{I}=\mathcal{J}'/\mathcal{M}$ and
\[
\bar{A}'/\bar{\mathcal{J}}' \cong (\bar{A}'/\bar{\mathcal{M}})/\bar{\mathcal{I}}\, .
\]
As the polynomial generators $W_1,\dots,W_{n-1}$ have become nilpotent in $A'/\mathcal{M}$, the ideal $(\pi(R_1'),\dots,\pi(R_{n-1}'))$ is a finitely generated $\Lambda_n$-module. Its saturation $\mathcal{I}$ can then be computed using CAS such as Macaulay2 \cite{Macaulay2}.

Fix $n\geq 1$. For $a,b\geq 0$ let $\mathscr{B}_{a,b}$ be an ordered collection of all monomials of weight $(a,b)$ in the variables $X_1,\dots,X_n$, $Y_1,\dots,Y_n$ and $W_1,\dots,W_{n-1}$. We have the following algorithm:
\begin{enumerate}
\item[(1)] For all weights $(a,b)$ with $1\leq a\leq n$ and $1\leq b \leq n$ do the following:
\begin{itemize}
\item[(i)] For all $M\in \mathscr{B}_{a-1,b-1}$ and all $i\in \{1,\dots,n-1\}$ expand $M R_i'$ as a $\Lambda_n$-linear combination of the basis elements in $\mathscr{B}_{a,b}$. Arrange the coefficient vectors into a matrix $T_{a,b}$ of size $|\mathscr{B}_{a,b}|\times (n-1)|\mathscr{B}_{a-1,b-1}|$ with entries in $\Lambda_n$.
\item[(ii)] The image of $T_{a,b}$ is a finitely generated $\Lambda_n$-submodule $\mathrm{im}(T_{a,b})\subseteq \Lambda_n^{\mathscr{B}_{a,b}}$. Let $G_{a,b}$ be the saturation of $\mathrm{im}(T_{a,b})$ with respect to $\Delta^2$.
\item[(iii)] Let $\varepsilon\co \Lambda_n^{\mathscr{B}_{a,b}}\to k^{\mathscr{B}_{a,b}}$ be reduction modulo $(e_1,\dots,e_n)$. Set $\bar{G}_{a,b}:=\varepsilon(G_{a,b})$.
\end{itemize}
\item[(2)] Return $\bar{G}_{a,b}$ for all $1\leq a,b \leq n$.
\end{enumerate}

For a monomial $M\in \mathscr{B}_{a,b}$, let $\bar{M}$ denote its image in $\bar{A}'$; so $\bar{M}$ is a monomial in the generators $\bar{X}_1,\dots,\bar{X}_n$, $\bar{Y}_1,\dots,\bar{Y}_n$ and $\bar{W}_1,\dots,\bar{W}_{n-1}$.

\begin{corollary} \label{cor:main2}
Let $k$ be a field of characteristic zero or prime to $n!$. Then the ideal $\bar{\mathcal{J}}'$ in Theorem \ref{thm:main2} is the $k$-linear span of
\begin{itemize}
\item[(i)] all $\bar{M}$ for $M\in \mathscr{B}_{a,b}$ with $a>n$ or $b> n$;
\item[(ii)] all $\sum_{M\in \mathscr{B}_{a,b}} f(M) \bar{M}$ for $f\in \bar{G}_{a,b}$ with $1\leq a,b \leq n$.
\end{itemize}
\end{corollary}

Of course, it suffices to draw $f$ from a basis of $\bar{G}_{a,b}$, so there are only finitely many of the relations in (ii) to be considered. Using Macaulay2 we compute the following example.

\begin{example}
The rational cohomology ring of $C_2(\mathrm{GL}_2(\C))$ is the quotient
\[
\left(\Q[\bar{W}_1] \otimes {\textstyle \bigwedge_\Q}(\bar{X}_1,\bar{Y}_1,\bar{X}_2,\bar{Y}_2)\right)/ K
\]
where $K$ is the $\Q$-linear span of
\begin{itemize}
\item monomials of weight $(a,b)$ with $a>2$ or $b>2$
\item in weight $(1,1)$: $\bar{X}_2 \bar{Y}_2$
\item in weight $(2,1)$:  $\bar{X}_1 \bar{X}_2 \bar{Y}_2$, $\bar{X}_1 \bar{X}_2 \bar{Y}_1+2\bar{X}_2 \bar{W}_1$
\item in weight $(1,2)$: $\bar{X}_2 \bar{Y}_1 \bar{Y}_2$, $\bar{X}_1 \bar{Y}_1 \bar{Y}_2-2\bar{Y}_2 \bar{W}_1$
\item in weight $(2,2)$: $\bar{X}_1 \bar{X}_2 \bar{Y}_1 \bar{Y}_2$,  $\bar{X}_1 \bar{Y}_2 \bar{W}_1$, $\bar{X}_2 \bar{Y}_1 \bar{W}_1$, $\bar{X}_2 \bar{Y}_2 \bar{W}_1$.
\end{itemize}
\end{example}

For $C_2(\mathrm{GL}_3(\C))$ the $\Q$-vectorspace of relations in bidegrees $(a,b)$ with $a\leq 3$ and $b\leq 3$ is already $150$-dimensional. It would be interesting to find an explicit generating set for the ideal of relations rather than the vectorspace of relations.

\begin{remark} \label{rem:kishimototakeda}
Kishimoto and Takeda \cite{KT21} have first shown that $H^\ast(C_2(U(n));k)$ is minimally generated by $3n-1$ elements when $k$ is a field of characteristic zero or prime to $n!$. But they were unable to determine the relations systematically. Let us explain the relationship between their generators and ours, and give an interpretation in terms of the Chern character. It follows from the formulae in (\ref{eq:explicitformxy}) and Lemma \ref{lem:newtonidentity} that the generators are the same:
\[
\bar{X}_l=z(l,\{1\}),\quad \bar{Y}_l=z(l,\{2\}),\quad \bar{W}_l=z(l,\{1,2\})
\]
for all $l$, where the notation $z(d,I)$ for their generators is introduced just before \cite[Lemma 6.7]{KT21}. Now recall from Section \ref{sec:generators} that there is a map
\[
C_2(U(n)) \to C_2(U)\simeq U^2\times BU
\]
such that primitive generators for $H^\ast(U^2;k)$ pull back to $\bar{X}_l$ and $\bar{Y}_l$, and the Chern classes in $H^\ast(BU;k)$ pull back to $\bar{Z}_l$ (the reduction modulo $(e_1,\dots,e_n)$ of the integral generators $Z_l$). If $c_1,c_2,\dots$ are the universal Chern classes, then $s_l(c_1,\dots,c_l)/l!=ch_l$ is the $l$-th component of the Chern character. Since $W_l$ is defined by $s_l(Z_1,\dots,Z_l)/l$, it follows that the generator $\bar{W}_l$ is the pullback of $(l-1)! ch_l \in H^{2l}(BU;k)$.
\end{remark}

\section{Integral surjectivity of the Atiyah-Bott map} \label{sec:ab}

As a by product of our proofs we obtain some interesting information on a type of map considered by Atiyah and Bott. Let $M=(S^1)^2$ and let $\mathcal{A}^{\flat}(n)$ denote the space of flat connections on the trivial principal $U(n)$-bundle over $M$. Let $\mathscr{G}=\mathrm{map}(M,U(n))$ denote the gauge group. The inclusion of $\mathcal{A}^{\flat}(n)$ into the affine space of all connections $\mathcal{A}(n)$ induces a map
\[
\alpha_n\co \mathcal{A}^{\flat}(n)_{h\mathscr{G}} \to \mathcal{A}(n)_{h\mathscr{G}} \simeq B\mathscr{G}\, .
\]

\begin{proposition} \label{prop:ab}
For every $n\geq 0$, the induced map
\[
\alpha_n^\ast\co H^\ast_{\mathscr{G}}(\mathcal{A}(n);\Z) \to H^\ast_{\mathscr{G}}(\mathcal{A}^{\flat}(n);\Z)
\]
is surjective.
\end{proposition}

In \cite{AB} Atiyah and Bott study this type of map for the space of flat connections on principal $U(n)$-bundles over a closed orientable surface of genus $g\geq 2$, in the case where the rank $n$ and the degree of the bundle are coprime. They show that the corresponding equivariant cohomology rings are torsion free and the map is surjective with rational coefficients (see \cite[Theorem 7.14]{AB}). In the genus $1$ case, we show that the map is surjective with integral coefficients. From this perspective, our main theorem (Theorem \ref{thm:main}) may be viewed as a description of the kernel of this map.

\begin{proof}
The map $\alpha_n$ has the following alternative description: The classifying space of $\mathscr{G}$ is the connected component of the mapping space $\mathrm{map}(M,BU(n))$ consisting of the null homotopic maps (see \cite[Proposition 2.4]{AB}). We denote it by
\[
B\mathscr{G} \simeq \mathrm{map}^0(M,BU(n))\, .
\]
Sending a representation $\pi_1(M)=\Z^2 \to U(n)$ to the classifying map $M\to BU(n)$ of the flat bundle it represents, gives a map up to homotopy
\[
C_2(U(n)) \to \mathrm{map}_\ast(M,BU(n))\,.
\]
Because the space $C_2(U(n))$ is connected, the image of this map is contained in the connected component of the nullhomotopic maps. Passing to homotopy orbits, and using the fact that
\[
\mathrm{map}^0_\ast(M,BU(n))_{hU(n)} \simeq \mathrm{map}^0(M,BU(n))\,,
\]
we obtain a map
\[
\alpha_n \co C_2(U(n))_{hU(n)} \to \mathrm{map}^0(M,BU(n)) \simeq B\mathscr{G}\, .
\]
The proposition claims that the induced map
\[
\alpha_n^\ast\co H^\ast(\mathrm{map}^0(M,BU(n)) ;\Z)\to H^\ast_{U(n)}(C_2(U(n);\Z)
\]
is surjective. Now there is a commutative diagram
\[
\xymatrix{
C_2(U(n))_{hU(n)} \ar[d] \ar[rr]^-{\alpha_n} && \mathrm{map}^0(M,BU(n)) \ar[d] \\
C_2(U)_{hU} \ar[rr]_-{\sim}^-{\alpha_{\infty}} && \mathrm{map}^0(M,BU) 
}
\]
in which the vertical maps are induced by the inclusion $U(n)\to U$. Ramras has proved that the bottom horizontal map is an equivalence (see \cite[Theorem 3.4]{RModuli}). The left vertical map is surjective in integral cohomology by Corollary \ref{cor:gcsurjective}. It follows that the top horizontal map is surjective in integral cohomology as well.
\end{proof}

\appendix

\section{$\Gamma$-spaces and the spectral sequence} \label{sec:appendix}

The purpose of this appendix is to explain the origin of the spectral sequence of Proposition \ref{prop:ss}, especially its multiplicative structure.

We recall the notion of a $\Gamma$-space (see e.g. the appendix of \cite{Global}). Let $\Topp$ be the category of compactly generated spaces. Let $\Fin_\ast$ be the category of finite based sets of the form $\langle n\rangle=\{\ast,1,\dots,n\}$, $n\geq 0$ with basepoint $\ast$.

\begin{definition}
A \emph{$\Gamma$-space} is a functor $F\co \Fin_\ast\to \Topp$ such that $F(\ast)=\ast$.
\end{definition}

For $n\geq 1$ and $1\leq i \leq n$ let $\rho^i\co \langle n\rangle \to \langle 1\rangle$ be the based map defined by $\rho^i(j)=1$ if $i=j$ and $\rho^i(j)=\ast$ if $i\neq j$.

\begin{definition}
A $\Gamma$-space $F\co \Fin_\ast\to \Topp$ is called \emph{special} if for every $n\geq 1$ the map $F(\langle n\rangle)\to F(\langle 1\rangle)^n$ induced by $\rho^1,\dots,\rho^n$ is a weak homotopy equivalence.
\end{definition}

Let $F$ be a special $\Gamma$-space, and let $\nabla\co \langle 2\rangle \to \langle 1\rangle$ denote the based map with $\nabla(1)=\nabla(2)=1$. Then there is a zig-zag 
\begin{equation} \label{eq:zigzag}
\xymatrix{
F(\langle 1\rangle)^2 && \ar[ll]_-{F(\rho^1)\times F(\rho^2)}^-{\sim} F(\langle \langle 2\rangle) \ar[r]^-{F(\nabla)} & F(\langle 1\rangle)
}
\end{equation}
in which the arrow pointing to the left is a weak equivalence. Since $\nabla$ is invariant under the $C_2$-action flipping $1,2\in \langle 2\rangle$, this zig-zag exhibits $F(\langle 1\rangle)$ as a commutative monoid in the homotopy category. Let $k$ be a commutative ring. On homology with coefficients in $k$ we obtain a product
\[
H_\ast(F(\langle 1\rangle);k)^{\otimes 2}\to H_\ast(F(\langle 1\rangle);k)\,,
\]
making $H_\ast(F(\langle 1\rangle);k)$ into a graded-commutative $k$-algebra. 

\begin{example} \label{ex:ku}
Let $\C^\infty$ be equipped with the standard Hermitian inner product. Let $\mathrm{Gr}(n,\C^\infty)$ denote the Grassmannian of $n$-dimensional complex subspaces of $\C^\infty$. Define $ku\co \Fin_\ast\to \Topp$ by setting $ku(\langle 1\rangle)=\bigsqcup_{n\geq 0} \mathrm{Gr}(n,\C^\infty)$ and
\[
ku(\langle n\rangle) \subseteq ku(\langle 1\rangle)^n
\]
the subspace of those $n$-tuples $(V_1,\dots,V_n)$ for which the $V_i$ are pairwise orthogonal. On a based map $\alpha\co \langle n\rangle\to \langle m\rangle$ define $ku(\alpha)\co ku(\langle n\rangle)\to ku(\langle m\rangle)$ by
\[
ku(\alpha)(V_1,\dots,V_n)=(W_1,\dots,W_m)\,,
\]
where $W_j=\bigoplus_{i\in \alpha^{-1}(j)} V_i$ (so $W_j=0$ if $\alpha^{-1}(j)=\emptyset$). This is a special $\Gamma$-space (see e.g. \cite[Theorem 6.3.19]{Global}). The underlying graded commutative $k$-algebra $H_\ast(ku(\langle 1\rangle);k)$ is a polynomial algebra $k[x_0,x_1,x_2,\dots]$, where $|x_i|=2i$.
\end{example}

A $\Gamma$-space $F$ can be prolonged to a functor on based simplicial sets: Let $X$ be a based simplicial set. We can evaluate $F$ levelwise on $X$ (by a filtered colimit in case $X_n$ is infinite). This gives a simplicial space $F(X)$ of which we can then take the geometric realisation $|F(X)|$.

The Loday construction (see Section \ref{sec:loday}) and $\Gamma$-spaces are related in the following way. Taking homology in every simplicial degree, we obtain a functor $H_\ast\circ F$ from $\sset_\ast$ into simplicial graded commutative $k$-algebras. 

\begin{lemma} \label{lem:functorhomology}
Let $F$ be a special $\Gamma$-space and $k$ a commutative ring such that $H_\ast(F(\langle 1\rangle);k)$ is flat as a $k$-module. Then there is a natural isomorphism
\[
\mathcal{L}(H_\ast(F(\langle 1\rangle);k);k) \cong H_\ast\circ F
\]
of functors from $\sset_\ast$ to simplicial graded commutative $k$-algebras.
\end{lemma}
\begin{proof}
It suffices to check this on finite pointed sets $\langle n\rangle\in \Fin_\ast$, where the isomorphism
\[
H_\ast(F(\langle n\rangle);k)\cong H_\ast(F\langle 1\rangle);k)^{\otimes n} \cong \mathcal{L}(H_\ast(F(\langle 1\rangle);k);k) (\langle n\rangle)
\]
is induced by $\rho^1,\dots,\rho^n$ using the K{\"u}nneth isomorphism.
\end{proof}

Let $X\in \sset_\ast$ be a finite based simplicial set, i.e., $X_n$ is finite for all $n$. Let $F$ be a special $\Gamma$-space. Then, for every $n$ there is a zig-zag
\begin{equation} \label{eq:zigzaggammaspace}
F(X_n)\times F(X_n) \xleftarrow{\sim} F(X_n\vee X_n) \to F(X_n)\,.
\end{equation}
The map pointing to the left is induced by the two maps $X_n\vee X_n\to X_n$ which are the identity on one summand and the constant map on the other summand; it is a weak homotopy equivalence, because $F$ is special. The map pointing to the right is induced by the fold map. The zig-zag is natural with respect to the simplicial structure, so it gives a zig-zag of simplicial spaces.

There is a technical condition on a $\Gamma$-space, called \emph{cofibrancy} (see \cite[Definition B.33]{Global}), which assures that $F(X)$ is a Reedy-cofibrant simplicial space.

\begin{example}
The $\Gamma$-space $ku$ from Example \ref{ex:ku} is cofibrant (see \cite[Example 6.3.16]{Global})
\end{example}

A levelwise weak equivalence of Reedy cofibrant simplicial spaces gives a weak equivalence on geometric realisations, so if $F$ is cofibrant there is a zig-zag
\[
|F(X)|\times |F(X)| \xleftarrow{\sim} |F(X\vee X)|\to |F(X)|\, ;
\]
taking homology we see that $H_\ast(|F(X)|;k)$ becomes a graded commutative $k$-algebra.

\begin{remark} \label{rem:underlyingspace}
If $X$ is a finite pointed simplicial set and $F$ is a cofibrant special $\Gamma$-space, then $\langle n\rangle \mapsto |F(X\wedge \langle n\rangle)|$ is again a cofibrant special $\Gamma$-space (\cite[Proposition B.54]{Global}), and $|F(X)|$ is just its underlying space.
\end{remark}

We expect that the following statement is well-known to the experts, but we were unable to find a reference that includes the statement about the multiplicativity of the spectral sequence.

\begin{proposition} \label{prop:ssgeneral}
Let $F$ be a cofibrant special $\Gamma$-space. View $H_\ast(F(\langle 1\rangle);k)$ as an augmented $k$-algebra via $F(\langle 1 \rangle)\to \ast$ and assume that it is a flat $k$-module. Let $X$ be a finite pointed simplicial set. Then there is a first-quadrant spectral sequence of algebras,
\[
E^2=\mathrm{HH}^X_\ast(H_\ast(F(\langle 1\rangle);k);k)\,\Longrightarrow\, H_\ast(|F(X)|;k)\, .
\]
A morphism of $\Gamma$-spaces $f\co F\to G$ induces a map of spectral sequences. On the $E^2$-page this map agrees with the map induced by the map of algebras
\[
f_\ast\co H_\ast(F(\langle 1\rangle);k)\to H_\ast(G(\langle 1\rangle);k)\, .
\]
\end{proposition}

The rest of the appendix is devoted to the proof of the proposition. Throughout $k$ denotes a commutative ring and homology groups are taken with $k$-coefficients. We first explain the content of \cite[Theorem 11.14]{Mayiteratedloopspaces}. Let $X$ be a simplicial space and $F_\ast |X|\subseteq |X|$ the skeletal filtration of its geometric realisation. To this filtration corresponds a spectral sequence
\[
E^1_{p,q}=H_{p+q}(F_p|X|,F_{p-1}|X|)\Longrightarrow H_{p+q}(|X|)\, .
\]
The differential $d^1\co E^1_{p,q}\to E^1_{p-1,q}$ is the boundary homomorphism in the long exact homology sequence of the triple $(F_p|X|,F_{p-1}|X|,F_{p-2}|X|)$.

For fixed $q$, let $H_q(X)$ be the simplicial $k$-module obtained from $X$ by taking $q$-th homology levelwise. As usual, we may view $H_q(X)$ as a chain complex of $k$-modules by taking the alternating sum of the face maps as the differential. For every $q$, there is a chain map
\[
f_q\co H_q(X) \to E^1_{\ast,q}
\]
which in degree $p$ is the composite
\[
H_q(X_p) \xrightarrow{\times \Delta^p} H_{p+q}(X_p\times \Delta^p,X_p\times \partial \Delta^p) \xrightarrow{\pi_\ast} E^1_{p,q}\, .
\]
The first map is the suspension isomorphism, given by the cross product with a $p$-simplex, and the second map is induced by the projection
\[
\pi\co X_p\times \Delta^p\to F_p|X|\,.
\]
If $X$ is Reedy cofibrant (or proper, in the sense of \cite[Definition 11.2]{Mayiteratedloopspaces}), then the restriction of $f_q$ to the normalised chain complex $NH_q(X_\ast) \subseteq H_q(X_\ast)$ is an isomorphism onto $E^1_{\ast,q}$, and hence $f_q$ is a quasi-isomorphism. The spectral sequence then takes the form
\[
E^2_{p,q}=H_p(H_q(X)) \Longrightarrow H_{p+q}(|X|)\, .
\]

Now let $Y$ be another simplicial space and $F_\ast|Y|\subseteq |Y|$ the skeletal filtration. We equip the product $|X|\times |Y|$ with the product filtration
\[
F_p(|X|\times |Y|)=\bigcup_{k+l=p} F_k|X|\times F_l|Y|\, .
\]
Let $E^r(|X|)$, $E^r(|Y|)$ and $E^r(|X|\times |Y|)$ denote the spectral sequences associated with the respective filtrations. The homology cross product and the inclusion $F_p|X|\times F_r|Y|\subseteq F_{p+r}(|X|\times |Y|)$ induce a natural cross product
\[
\times\co H_{p+q}(F_p|X|)\otimes H_{r+s}(F_r|Y|) \to H_{p+q+r+s}(F_{p+r}(|X|\times |Y|))\, .
\]
Passing to relative homology, this induces a cross product
\[
\times\co E^1_{p,q}(|X|)\otimes E^1_{r,s}(|Y|) \to E^1_{p+r,q+s}(|X|\times |Y|)\, . 
\]

\begin{lemma} \label{lem:pairing1}
There is a pairing of spectral sequences
\[
E^r(|X|) \otimes E^r(|Y|) \to E^r(|X|\times |Y|)\, .
\]
On $E^1$-pages the pairing is given by the $\times$-product, and on $E^\infty$-pages it coincides with the pairing induced by the cross product
\[
\times\co H_{p+q}(|X|) \otimes H_{r+s}(|Y|)\to H_{p+q+r+s}(|X|\times |Y|)\, .
\]
\end{lemma}
\begin{proof}
This follows from the well-known fact that a pairing of filtered complexes gives rise to a pairing of spectral sequences. Let $C_\ast$ denote the singular $k$-chain complex functor. Then one only needs to observe that the Eilenberg-Zilber map (which induces the $\times$-product)
\[
C_\ast(|X|)\otimes C_\ast(|Y|)\to C_\ast(|X|\times |Y|)
\]
is filtration preserving, if the singular chain complexes are given the filtrations induced from those of $|X|$, $|Y|$ and $|X|\times |Y|$, respectively, and the tensor product complex is given the obvious product filtration.
\end{proof}

The product simplicial space $X\times Y$ has $p$-simplices
\[
(X\times Y)_p=X_p\times Y_p\,.
\]
If $|X\times Y|$ is filtered by its skeleta and $|X|\times |Y|$ has the product filtration, then the homeomorphism $|X\times Y| \cong |X|\times |Y|$ induced by the two projections is not filtration preserving in general, but its inverse
\[
\zeta\co |X|\times |Y|\xrightarrow{\cong} |X\times Y|
\]
is always filtration preserving (see \cite[Lemma 11.15]{Mayiteratedloopspaces}). Therefore, it induces a map of spectral sequences
\[
E^r(\zeta)\co E^r(|X|\times |Y|) \to E^r(|X\times Y|)\, .
\]
By composing $E^r(\zeta)$ with the pairing of Lemma \ref{lem:pairing1}, we get a pairing of spectral sequences
\[
E^r(|X|)\otimes E^r(|Y|)\to E^r(|X\times Y|)\, .
\]
On $E^1$-pages this pairing is given by $E^1(\zeta)\circ \times$ and on $E^\infty$-pages it coincides with pairing induced by
\[
\zeta_\ast\circ \times\co H_{p+q}(|X|)\otimes H_{r+s}(|Y|) \to H_{p+q+r+s}(|X\times Y|)\, .
\]

Define a shuffle map
\[
\mathrm{sh} \co H_q(X_p)\otimes H_s(Y_r) \to  H_{q}(X_{p+r}) \otimes H_s(Y_{p+r}) 
\]
by
\[
c \otimes d \mapsto \sum_{(\mu,\nu)\in \mathrm{Sh}(p,r)} \mathrm{sgn}(\mu,\nu)\, s_\nu^X(c)\otimes  s_\mu^Y(d)\, .
\]
Here $\mathrm{Sh}(p,r)\subseteq \Sigma_{p+r}$ is the set of $(p,r)$-shuffles, $\mathrm{sgn}(\mu,\nu)$ is the sign of the shuffle permutation $(\mu,\nu)$, and $s_\nu^X=s_{\nu_r}^X \cdots s_{\nu_1}^X$ and $s_\mu^Y=s_{\mu_p}^Y \cdots s_{\mu_1}^Y$ are compositions of degeneracy maps in the simplicial $k$-modules $H_q(X)$ and $H_s(Y)$, respectively.

\begin{lemma} \label{lem:pairing2}
The pairing of spectral sequences
\[
E^r(|X|)\otimes E^r(|Y|)\to E^r(|X\times Y|)
\]
makes the following diagram commute
\[
\xymatrix{
H_q(X_p)\otimes H_s(Y_r) \ar[d]^-{f_{q,p}\otimes f_{s,r}} \ar[r]^{\times\circ \mathrm{sh}} & H_{q+s}((X\times Y)_{p+r}) \ar[d]^-{f_{q+s,p+r}} \\
E^1_{p,q} \otimes E^1_{r,s} \ar[r]^-{E^1(\zeta)\circ \times} & E^1_{p+q,r+s}.
}
\]
\end{lemma}
\begin{proof}
This is a diagram chase. Recall that the $(q+s)$-simplices in the standard triangulation of $\Delta^q\times \Delta^s$ are represented by maps $l_{(\mu,\nu)}\co \Delta^{q+s}\to \Delta^q\times \Delta^s$ which correspond to shuffle permutations $(\mu,\nu)\in \mathrm{Sh}(q,s)$ (see \cite[Section 3.B]{Hatcher}). Chasing simplices $\sigma\co \Delta^q\to X_p$ and $\tau\co \Delta^s\to Y_r$ through the top right corner of the diagram, we obtain an alternating sum over shuffles $(\mu,\nu)\in \mathrm{Sh}(q,s)$, $(\mu',\nu')\in \mathrm{Sh}(p,r)$ and $(\mu'',\nu'')\in \mathrm{Sh}(q+s,p+r)$ of simplices
\[
\Delta^{q+s+p+r}\to F_{p+r}|X\times Y|
\]
of the form
\begin{equation} \label{eq:diagramchase1}
\pi\circ ((s^X_{\nu'}(\sigma)\times s_{\mu'}^Y(\tau))\circ l_{(\mu,\nu)}\times id_{p+r})\circ l_{(\mu'',\nu'')}\, .
\end{equation}
Chasing $\sigma$ and $\tau$ through the bottom left corner, we obtain an alternating sum over shuffles $(\alpha,\beta)\in \mathrm{Sh}(q,p)$, $(\alpha',\beta')\in \mathrm{Sh}(s,r)$ and $(\alpha'',\beta'')\in \mathrm{Sh}(q+p,s+r)$ of simplices of the form
\[
\zeta\circ (\pi\times \pi)\circ (\sigma\times id_p \times \tau \times id_r)\circ (l_{(\alpha,\beta)}\times l_{(\alpha',\beta')}) \circ l_{(\alpha'',\beta'')}\, .
\]
Now we observe that
\[
 (l_{(\alpha,\beta)}\times l_{(\alpha',\beta')}) \circ l_{(\alpha'',\beta'')}= (id\times T\times id)\circ (l_{(\mu,\nu)}\times l_{(\mu',\nu')}) \circ l_{(\mu'',\nu'')}
\]
for uniquely determined $(\mu,\nu)\in \mathrm{Sh}(q,s)$, $(\mu',\nu')\in \mathrm{Sh}(p,r)$ and $(\mu'',\nu'')\in \mathrm{Sh}(q+s,p+r)$ and the transposition $T\co \Delta^p\times \Delta^s\to \Delta^s\times \Delta^p$. Using this we can rewrite the previous sum as a sum over simplices of the form
\[
\zeta\circ(\pi\times \pi)\circ (id\times T\times id)  \circ ((\sigma\times \tau)\circ l_{(\mu,\nu)}\times (id_p\times id_r)\circ l_{(\mu',\nu')})\circ l_{(\mu'',\nu'')}\, .
\]
At this point we can insert the definition of $\zeta$ from \cite[Theorem 11.5]{Mayiteratedloopspaces} which transforms the given expression precisely into the one of (\ref{eq:diagramchase1}).
\end{proof}

From now on we let $X$ be a finite based simplicial set. Let $F$ be a $\Gamma$-space and denote by $F(X)$ the simplicial space obtained by evaluating $F$ levelwise. The zig-zag (\ref{eq:zigzaggammaspace}) induces a zig-zag of maps of spectral sequences
\begin{equation} \label{eq:zigzagss}
E^r(|F(X)\times F(X)|) \leftarrow E^r(|F(X\vee X)|) \rightarrow E^r(|F(X)|)
\end{equation}
as well as a diagram of chain complexes
\[
\xymatrix{
H_q(F(X)\times F(X)) \ar[d]^-{f_{q}} & \ar[l] H_q(F(X\vee X)) \ar[d]^-{f_{q}} \ar[r] & H_q(F(X)) \ar[d]^-{f_{q}} \\
E^1_{\ast,q}(|F(X)\times F(X)|) & \ar[l] E^1_{\ast,q}(|F(X\vee X)|) \ar[r] & E^1_{\ast,q}(|F(X)|).
}
\]
The diagram commutes, because the chain maps $f_q$ are natural with respect to maps of simplicial spaces.

If $F$ is special, then the top left horizontal arrow in the diagram becomes an isomorphism of chain complexes for every $q\geq 0$. If $F$ is also cofibrant, then $F(X)$ is Reedy cofibrant (see \cite[Proposition B.37]{Global}), and so is $F(X)\times F(X)$ by \cite[Proposition A.50]{Global}. Hence, the vertical maps in the diagram are quasi-isomorphisms. Thus, from the $E^2$-page onwards we can invert the leftward pointing arrow in (\ref{eq:zigzagss}). By combining with the pairing of Lemma \ref{lem:pairing2} we obtain a pairing of spectral sequences
\[
E^r(|F(X)|)\otimes E^r(|F(X)|) \to E^r(|F(X)|)
\]
provided $r\geq 2$. By construction, the product on $E^\infty$-pages coincides with the one induced by the product on $H_\ast(|F(X)|)$.

To finish the proof it remains to show that there is an isomorphism of bigraded algebras
\[
E^2(|F(X)|) \cong \HH^{X}_\ast (H_\ast(F(\langle 1\rangle));k)\, .
\]
We recall the definition of the shuffle product in Hochschild homology.

\begin{definition} \label{def:shuffle}
Let $A$ be an augmented dgca over $k$. The \emph{shuffle product} is the pairing of bigraded commutative bidifferential algebras
\[
\mathcal{L}(A,k)(X)\otimes \mathcal{L}(A,k)(X)\to \mathcal{L}(A,k)(X)
\]
defined as the composition
\[
 \mathcal{L}(A,k)(X_p)\otimes \mathcal{L}(A,k)(X_q)\xrightarrow{\mathrm{sh}} \mathcal{L}(A\otimes A,k\otimes k)(X_{p+q}) \xrightarrow{\textnormal{mult.}}  \mathcal{L}(A, k)(X_{p+q})\, ;
\]
the first map is defined for $a\in \mathcal{L}(A,k)(X_p)$ and $b\in \mathcal{L}(A,k)(X_q)$ by
\[
\mathrm{sh}(a,b):= \sum_{(\mu,\nu)\in \mathrm{Sh}(p,q)} \mathrm{sgn}(\mu,\nu) (s_{\nu_q}\cdots s_{\nu_1})(a) \otimes (s_{\mu_p}\cdots s_{\mu_1})(b)\,,
\]
where the $s_{\nu_i}$'s and the $s_{\mu_i}$'s are the degeneracy maps in the simplicial dgca $\mathcal{L}(A,k)(X)$, and the second map is induced by the multiplication in $A$ and $k$.
\end{definition}

By Lemma \ref{lem:functorhomology}, there is an isomorphism of simplicial graded commutative $k$-algebras
\begin{equation} \label{eq:sshochschildcomplex}
H_\ast(F(X_p))\cong \mathcal{L}(H_\ast(F(\langle 1 \rangle));k)(X_p)\, .
\end{equation}
By Lemma \ref{lem:pairing2}, the product on $E^2(|F(X)|)$ is induced, upon taking homology, by the composite of
\[
H_q(F(X_p))\otimes H_s(F(X_r))\xrightarrow{\times\circ \mathrm{sh}} H_{q+s}((F(X)\times F(X))_{p+r}) 
\]
with the multiplication
\[
H_{q+s}((F(X)\times F(X))_{p+r}) \xleftarrow{\cong} H_{q+s}(F(X\vee X)_{p+r})  \to H_{q+s}(F(X)_{p+r})\, .
\]
Under the isomorphism (\ref{eq:sshochschildcomplex}) it corresponds to the shuffle product on the Loday construction $\mathcal{L}(H_\ast(F(\langle 1\rangle));k)(X)$.

The naturality statement of Proposition \ref{prop:ssgeneral} is clear.

\bibliography{refs}
\bibliographystyle{plain}

\end{document}